\documentclass[12pt]{amsart}
\usepackage{amssymb,amsmath,amsthm,latexsym}
\usepackage{mathrsfs}
\usepackage{mdwlist,enumitem}
\usepackage{a4wide}
\usepackage{graphicx}
\usepackage{mathtools,dsfont}
\usepackage{mathdots}

\newcommand{\N}{\mathbb{N}}

\newcommand{\p}{\varphi}
\newcommand{\e}{\varepsilon}
\newcommand{\w}{\widetilde}
\newcommand{\oo}{\overline}
\newcommand{\EE}{\mathfrak{E}}

\newcommand{\FF}{\mathfrak{F}}

\newcommand{\TT}{\mathfrak{T}}
\newcommand{\ind}{\mathds{1}}
\newcommand{\Sz}{\mathrm{Sz}}
\newcommand{\Cz}{\mathrm{Cz}}
\newcommand{\BM}{\mathsf{BM}}
\newcommand{\ssp}{\scriptscriptstyle{p}}
\newcommand{\ssq}{\scriptscriptstyle{q}}

\newcommand{\abs}[1]{\lvert#1\rvert}
\newcommand{\n}[1]{\|#1\|}

\newcommand{\supp}{\mathrm{supp}}
\newcommand{\ccup}{\scalebox{0.95}{$\bigcup$}}

\renewcommand{\leq}{\leqslant}
\renewcommand{\geq}{\geqslant}
\renewcommand{\frown}{\mathbin{\raisebox{0.3ex}{$\smallfrown$}}}

\newcommand{\vertiii}[1]{{\left\vert\kern-0.25ex\left\vert\kern-0.25ex\left\vert #1 
    \right\vert\kern-0.25ex\right\vert\kern-0.25ex\right\vert}}
\newcommand{\vertiiib}[1]{{\Biggl\vert\kern-0.25ex\Biggl\vert\kern-0.25ex\Biggl\vert #1 
    \Biggr\vert\kern-0.25ex\Biggr\vert\kern-0.25ex\Biggr\vert}}

\newtheorem{theorem}{Theorem}[section]
\newtheorem{lemma}[theorem]{Lemma}

\newtheorem{proposition}[theorem]{Proposition}
\newtheorem{corollary}[theorem]{Corollary}

\theoremstyle{definition}

\newtheorem{example}[theorem]{Example}

\theoremstyle{remark}
\newtheorem{remark}[theorem]{Remark}

\numberwithin{equation}{section}

\begin{document}
\title[Direct sums and summability of the Szlenk index]{Direct sums and summability of the Szlenk index}

\author[Sz. Draga]{Szymon Draga}
\address{Institute of Mathematics, University of Silesia, Bankowa 14, 40-007 Katowice, Poland}
\email{szymon.draga@gmail.com}

\author[T. Kochanek]{Tomasz Kochanek}
\address{Institute of Mathematics, Polish Academy of Sciences, \'Sniadeckich 8, 00-656 Warsaw, Poland\, {\rm and}\, Institute of Mathematics, University of Warsaw, Banacha~2, 02-097 Warsaw, Poland}
\email{tkoch@impan.pl}


\begin{abstract}
We prove that the $c_0$-sum of separable Banach spaces with uniformly summable Szlenk index has summable Szlenk index, whereas this result is no longer valid for more general direct sums. We also give a~formula for the Szlenk power type of the \mbox{$\EE$-direct} sum of separable spaces provided that $\EE$ has a~shrinking unconditional basis whose dual basis yields an~asymptotic $\ell_p$ structure in $\EE^\ast$. As a~corollary, we show that the Tsirelson direct sum of infinitely many copies of $c_0$ has power type $1$ but non-summable Szlenk index.
\end{abstract}
\maketitle

\section{Introduction}
Through recent years, various ordinal indices has become one of the main tools in studying geometry and structural properties of Banach spaces. The starting point was the notion of Szlenk index, introduced in \cite{szlenk} in order to show that there is no separable reflexive Banach space universal for the class of all separable reflexive Banach spaces. Also some quantitative considerations concerning the Szlenk index, especially its summability and the Szlenk power type, lead to significant development in understanding the geometry of Banach spaces with Szlenk index at most $\omega$; see \cite{gkl} and the survey article \cite{lancien}.

In this paper, we study the behavior of summability of the Szlenk index and the Szlenk power type with respect to general direct sums. Our study may be regarded as a~natural continuation of P.A.H.~Brooker's work (\cite{brooker}) concerning Szlenk indices of direct sums of Banach spaces and operators acting on them.

Summability of the Szlenk index may be understood as a~`strong' separability condition on the dual space. Banach spaces sharing this property resemble subspaces of $c_0$ in the sense that they admit `almost' Lipschitz UKK$^\ast$-renormings, whereas Banach spaces having a~truly Lipschitz UKK$^\ast$-renorming are just subspaces of $c_0$ ({\it cf. }\cite{gkl} and \cite{gkl_gafa}). Nevertheless, spaces with summable Szlenk index form an~interesting class of Banach spaces, where the original Tsirelson space serves as a~reflexive example. More generally, all Banach spaces with duals admiting a~finite-dimensional decomposition with a~skipped asymptotic $\ell_1$ blocking have summable Szlenk index ({\it cf. }\cite{kos}). As pointed out in \cite{gkl}, it is the `lack of isotropy' of the Szlenk derivation which is responsible for this quite an~abundance of Banach spaces with summable Szlenk index. Somewhat surprisingly, this class is also relevant to the problem of the existence of non-trivial twisted sums of $C(K)$-spaces (\cite{ccky}).

Rather than structural, there are some renorming properties which accurately reflect the nature of Banach spaces whose Szlenk index is summable or has a~prescribed power type. The problems here considered thus require combining standard techniques, employed when dealing with finite dimensional decompositions, with certain geometric tools. Regarding this latter aspect, we heavily rely on the seminal paper by G.~Godefroy, N.J.~Kalton and G.~Lancien \cite{gkl}. 

Our main results are: {\bf (1)} the $c_0$-sum of separable Banach spaces with uniformly summable Szlenk index has summable Szlenk index, however, this result fails to hold when $c_0$ is replaced by a~general Banach space with summable Szlenk index, {\it e.g. }the Tsirelson space (see Theorem~\ref{c0-sum} and Example~\ref{T(c0)}); {\bf (2)} under reasonable assumptions on $\EE$, the Szlenk power type of the $\EE$-direct sum of separable Banach spaces $X_n$ ($n\in\N$) is as expected, that is, the maximum of the Szlenk power type of $\EE$ and the `best' upper bound---understood in an~appropriate way---of Szlenk power types of $X_n$'s (see Theorem~\ref{power}). This result naturally involves asymptotic $\ell_p$ spaces, as explained below. As a~corollary, we prove that $\TT(c_0)$, the Tsirelson direct sum of infinitely many copies of $c_0$, is a~Banach space with Szlenk power type $1$ but non-summable Szlenk index; it is probably the first such example in the literature.

The organization of the paper is as follows. In Section~2 we introduce all necessary definitions concerning Szlenk derivation, Szlenk power type and tree-maps in Banach spaces. We also give outlines of proofs of two renorming theorems which basically come from \cite{gkl} and have been adjusted for future use. In Section~3 we prove our first main result listed above. Section~4 is a~kind of supplement to the previous one; we prove the summability of the Szlenk index for $\EE$-direct sums of finite-dimensional spaces, where $\EE$ is any Banach space with an~unconditional basis and summable Szlenk index. Although it can be also derived from a~result by H.~Knaust, E.~Odell and Th.~Schlumprecht~\cite{kos}, we present a~slightly different approach. The final section is mainly devoted to the proof of Theorem~\ref{power} which gives a~formula for the Szlenk power type of the direct sum $(\bigoplus_{n=1}^\infty X_n)_{\EE}$ in the case where $\EE^\ast$ is asymptotic $\ell_p$ and the Szlenk power types of $X_n$'s are appropriately bounded. The key step relies on deriving a~log-type estimate for the norm of an `$\ell_p$-sum' of positive vectors in an~asymptotic $\ell_p$ space (see Proposition~\ref{log_estimate}). As a~by-product, we obtain an~asymptotically sharp lower estimate for Tsirelson's norm of the sum of positive vectors in terms of the sum of their norms, where no assumption on the supports is made. We end the paper by giving two counterexamples which show that Theorem~\ref{power} is, in a~sense, the best one could expect.

\section{Tree-maps and renormings}
Let $X$ be a Banach space, $K$ be a~weak$^\ast$-compact subset of $X^\ast$ and $\e>0$. We define the $\e$-{\it Szlenk derivation} of $K$ by
$$
\iota_\e K=\bigl\{x^\ast\in K\colon \mathrm{diam}(K\cap U)>\e\mbox{ for every }w^\ast\mbox{-open neighborhood of }x^\ast\bigr\}
$$
and its iterates by $\iota_\e^0K=K$, $\iota_\e^{\alpha+1}K=\iota_\e(\iota_\e^\alpha K)$ for any ordinal $\alpha$, and $\iota_\e^\alpha K=\bigcap_{\beta<\alpha}\iota_\e^\beta K$ for any limit ordinal $\alpha$. The $\e$-{\it Szlenk index} of $X$, $\Sz(X,\e)$, is defined as the least ordinal $\alpha$ (if any such exists) for which $\iota_\e^\alpha B_{X^\ast}=\varnothing$, where $B_{X^\ast}$ is the unit ball in $X^\ast$. The {\it Szlenk index} of $X$ is defined as $\Sz(X)=\sup_{\e>0}\Sz(X,\e)$. Following \cite{gkl} we also define $\hat\iota_\e K$ as the weak$^\ast$-closed convex hull of $\iota_\e K$; $\Cz(X,\e)$ and $\Cz(X)$ are defined analogously as above.

In this paper, we are exclusively concerned with the case where $\Sz(X)\leq\omega$ which is equivalent to $\Sz(X,\e)$ being finite for every $\e>0$. The function $\Sz(X,\cdot)$ is then submultiplicative and hence there exists a~limit
$$
p(X)\coloneqq\lim_{\e\to 0+}\frac{\log\Sz(X,\e)}{\abs{\log\e}},\quad 1\leq p(X)<\infty,
$$
which defines the {\it Szlenk power type} of $X$. Equivalently, we have
\begin{equation}\label{p_def}
p(X)=\inf\bigl\{q\geq 1\colon \sup_{0<\e<1}\e^q\;\!\Sz(X,\e)<\infty\bigr\}
\end{equation}
({\it cf. }\cite{lancien} for further information). Note that in general the infimum may not be attained, however, for every $\delta>0$ there exists a~constant $C>0$ such that $\Sz(X,\e)\leq C\e^{-p(X)-\delta}$. Observe also that the Szlenk power type is an~isomorphic invariant. Indeed, if $d$ equals the Banach--Mazur distance between $X$ and $Y$, then an elementary calculation shows that ${\Sz(X,d\e)\leq\Sz(Y,\e)}$ ({\it cf. }\cite[Lemma 2.3]{gkl}). Hence,
$$
p(Y)=\lim_{\e\to 0+}\frac{\log\Sz(Y,\e)}{\abs{\log\e}}\geq\lim_{\e\to 0+}\frac{\log\Sz(X,d\e)}{\abs{\log\e}}=p(X).
$$

We say that $X$ has {\it summable Szlenk index} provided that there is a~constant $M$ such that for all positive $\e_1,\ldots,\e_n$ we have 
$$
\sum_{i=1}^n\e_i\leq M\,\,\mbox{ whenever }\,\,\iota_{\e_1}\ldots\iota_{\e_n}B_{X^\ast}\not=\varnothing.
$$
Then we also say that $X$ has summable Szlenk index {\it with constant $M$}. Given any family of Banach spaces, we shall say that they have {\it uniformly summable Szlenk index} provided that all of them have summable Szlenk index with the same constant.

One of the most profitable ways of handling Szlenk derivations is to consider tree-maps with values in a~Banach space. Here, we shall adopt the approach of Godefroy, Kalton and Lancien \cite{gkl}. Their two renorming theorems (in a~slightly modified form) will be of particular importance for us.

Let $\mathcal{F}\N$ stand for the family of all finite subsets of $\N$, usually written in increasing manner, and equipped with the following partial order relation: if $a=\{m_1,\ldots,m_j\}$ and $b=\{n_1,\ldots,n_k\}$ with $m_1<\ldots<m_j$ and $n_1<\ldots<n_k$, then $a\leq b$ if and only if $j\leq k$ and $m_i=n_i$ for each $1\leq i\leq j$. For any $a\in\mathcal{F}\N$, we denote $\abs{a}$ the cardinality of $a$; $b\in\mathcal{F}\N$ is called a~{\it successor} of $a$ if $\abs{b}=\abs{a}+1$ and $a\leq b$, so $b=a\!\frown\! n$ for some $n\in\N$, where $\frown$ stands for the concatenation operation. We say that a~subset $S\subseteq\mathcal{F}\N$ is a~{\it full subtree} provided that:
\begin{itemize}
\item $\varnothing\in S$;
\item each $a\in S$ has infinitely many successors in $S$;
\item if $a\in S$ and $a\not=\varnothing$, then the unique direct predecessor of $a$ belongs to $S$.
\end{itemize}
For simplicity, we shall use the letter $S$ for $\mathcal{F}\N$ henceforth. If $T\subseteq S$ is a~full subtree, every sequence $\beta=(a_n)_{n=0}^N$,  so that $a_0=\varnothing$ and $a_{n+1}$ is a~successor of $a_n$ for $0\leq n<N$ will be called a~{\it branch} in the case where $N=\infty$, and a~{\it partial branch} when $N<\infty$.

Let $V$ be a vector space and $T\subseteq S$ be a~full subtree. By a~{\it tree-map in $V$} we mean any map $a\mapsto x_a\in V$ defined on $T$ such that $x_{\varnothing}=0$ and the set $\{a\in\beta\colon x_a\not=0\}$ is finite for every branch $\beta\subset T$. Any such map will be typically denoted by $(x_a)_{a\in T}$. Given a~topology $\tau$ on $V$ we shall say that $(x_a)_{a\in T}$ is {\it $\tau$-null} provided that for every $a\in T$ the sequence $(x_{a\smallfrown n}\colon n\in\N,\, a\!\frown\! n\in T)$ is $\tau$-null. In particular, if $X$ is a~Banach space, we may consider weakly null tree-maps in $X$ and weak$^\ast$-null tree-maps in $X^\ast$.

\begin{lemma}[{\it cf. }{\cite[Prop.~3.4]{gkl}}]\label{gkl_lemma}
Let $X$ be a separable Banach space and $\e_1,\ldots,\e_n>0$. In order that $\iota_{\e_1}\ldots\iota_{\e_n}B_{X^\ast}\not=\varnothing$ it is necessary that there exists a~weak$^\ast$-null tree-map $(x_a^\ast)_{a\in S}$ in $X^\ast$ such that $\n{x_a^\ast}\geq\frac{1}{4}\e_{\abs{a}}$ for $1\leq\abs{a}\leq n$ and $\n{\sum_{a\in\beta}x_a^\ast}\leq 1$ for every branch $\beta\subset S$, and it is sufficient that there exists a~weak$^\ast$-null tree-map $(x_a^\ast)_{a\in S}$ in $X^\ast$ such that $\n{x_a^\ast}\geq\e_{\abs{a}}$ for $1\leq\abs{a}\leq n$ and $\n{\sum_{a\in\beta}x_a^\ast}\leq 1$ for every branch $\beta\subset S$. 
\end{lemma}

Now, we recall some important concepts introduced by Godefroy, Kalton and Lancien which connect tree-maps with renorming problems and then we formulate two renorming theorems coming from \cite{gkl}. We modify them a~bit in order to prepare the ground for future applications, so brief proofs are included for the sake of completeness. The first statement is an elaboration of a~part of \cite[Theorem~4.10]{gkl}---the important issue is that an~occurring number $c$ depends only on the constant of summability of Szlenk index. The second statement is a~more quantitative version of \cite[Theorem~4.8]{gkl} and exhibits the `$\ell_q$-like' behaviour in the dual of a~space with Szlenk power type less than $q$.

We say that a tree-map $(x_a)_{a\in T}$ is {\it of height} $n\in\N$ if $n$ is the largest integer for which there exists $a\in T$ with $\abs{a}=n-1$ and $x_a\not=0$. Given a~separable Banach space $X$ and $\sigma>0$, we define $N=N(\sigma)$ to be the least integer $N$ for which there exists a~weakly null tree-map $(x_a)_{a\in S}$ in $X$ of height $N+1$ such that $\n{x_a}\leq\sigma$ for each $a\in S$ and $\n{\sum_{a\in\beta}x_a}>1$ for each branch $\beta\subset S$, and we set $N(\sigma)=\infty$ if none such tree-map exists.

For any continuous, monotone increasing functions $f,g$ on $[0,1]$ satisfying $f(0)=g(0)=0$ we say that $f$ $C$-{\it dominates} $g$ if $f(\tau)\geq g(\tau/C)$ for every $\tau\in [0,1]$, and then we write $f\gtrsim_{\,C} g$ (or $g\lesssim_{\,C} f$). If $f\gtrsim_{\,C} g$ and $f\lesssim_{\,C} g$, then we say that $f$ and $g$ are $C$-{\it equivalent} and we write $f\simeq_{\,C}g$.

For any function $f$ as above we define its {\it dual Young's function} $f^\ast$ by $$f^\ast(\sigma)=\sup\bigl\{\sigma\tau-f(\tau)\colon 0\leq\tau\leq 1\bigr\},$$
and we note that $f^\ast$ is always a~convex function and $f^\ast\lesssim_{\,C}g^\ast$ whenever $f\gtrsim_{\,C}g$.

The following two functions introduced in \cite{gkl} give a~proper language for the proofs of the aforementioned renorming theorems:
\begin{itemize}
\item $\p(\sigma)=\inf\bigl\{\rho_{Y}(\sigma)\colon d_\BM(X,Y)\leq 2\bigr\}$,\\ where $\rho_Y(\sigma)$ is the least constant $\rho\geq 0$ so that if $y,y_n\in Y$ ($n\in\N$) satisfy $\n{y}=1$, $y_n\xrightarrow[]{\,w\,}0$ and $\n{y_n}\leq\sigma$, then
$\limsup_{n\to\infty}\n{y+y_n}\leq 1+\rho$;

\vspace*{2mm}
\item $\psi(\tau)=\sup\bigl\{\theta_Y(\tau)\colon d_\BM(X,Y)\leq 2\bigr\}$,\\
where $\theta_Y(\tau)$ is the greatest constant $\theta\geq 0$ so that if $y^\ast,y_n^\ast\in Y^\ast$ ($n\in\N$) satisfy $\n{y}=1$, $y_n^\ast\xrightarrow[]{\,w\ast\,}0$ and $\n{y_n^\ast}\geq\tau$, then $\liminf_{n\to\infty}\n{y^\ast+y_n^\ast}\geq 1+\theta$. 
\end{itemize}
(Here, $d_\BM$ stands for the Banach--Mazur distance.) According to \cite[Prop.~2.8]{gkl}, we have
\begin{equation}\label{ppsi}
\p\simeq_{\,8}\psi^\ast\qquad\mbox{ and }\qquad \p^\ast\simeq_{\,4}\psi.
\end{equation}

\begin{proposition}\label{renorming}
Let $X$ be a separable Banach space having summable Szlenk index with constant $M$. Then there is a~number $c=c(M)\in (0,1)$ depending only on $M$ such that for every $\tau\in (0,1)$ there exists a~norm $\abs{\,\cdot\,}$ on $X$ having the following properties (we use the symbol $\abs{\,\cdot\,}$ also for the corresponding dual norm):
\begin{itemize}
\item[{\rm (i)}] $\frac{1}{2}\n{x}\leq\abs{x}\leq\n{x}$ for every $x\in X$;

\vspace*{1mm}
\item[{\rm (ii)}] if $x^\ast\in X^\ast$, $\abs{x^\ast}=1$ and $(x_n^\ast)_{n=1}^\infty\subset X^\ast$ is a~weak$^\ast$-null sequence with $\abs{x_n^\ast}\geq\xi\geq\tau$ for $n\in\N$, then $$\liminf_{n\to\infty}\abs{x^\ast+x_n^\ast}\geq 1+c\xi.$$
\end{itemize}
\end{proposition}
\begin{proof}[Proof outline]
First, observe that it is sufficient to prove condition (ii) only for $\xi=\tau$. Indeed, assuming that (ii) is valid for $\xi=\tau$ we have
\begin{equation*}
\begin{split}
\liminf_{n\to\infty}\abs{x^\ast+x_n^\ast} & \geq\liminf_{n\to\infty}\Big|\frac{\xi}{\tau}x^\ast+x_n^\ast\Big|-\Big(\frac{\xi}{\tau}-1\Big)\abs{x^\ast}\\
& =\frac{\xi}{\tau}\liminf_{n\to\infty}\Big|x^\ast+\frac{\tau}{\xi}x_n^\ast\Big|-\frac{\xi}{\tau}+1\geq 1+c\xi.
\end{split}
\end{equation*}

Suppose that $\sigma\in (0,1)$ and $N=N(\sigma)$ is finite. Then, by \cite[Lemma~4.3]{gkl}, there exist $\e_1,\ldots,\e_N\in (0,1)$ so that $\iota_{\e_1}\ldots\iota_{\e_N}B_{X^\ast}\not=\varnothing$ and $\sum_{k=1}^N\e_k>(3\sigma)^{-1}$. Therefore, $N(\sigma)=\infty$ for every $\sigma<(3M)^{-1}$. By \cite[Thm.~4.4]{gkl}, we have $N(\sigma)^{-1}\simeq_{\,C}\p(\sigma)$ for some absolute constant $C\leq 19200$, which implies that $\p(\sigma)=0$ for $\sigma<(3CM)^{-1}$, and hence
$$
\p^\ast(\tau)\geq\sup\bigl\{\tau\sigma\colon 0\leq\sigma<(3CM)^{-1}\bigr\}=(3CM)^{-1}\tau.
$$
Now, \eqref{ppsi} yields that $\psi(\tau)\geq (12CM)^{-1}\tau$ for every $\tau\in [0,1]$. This ensures the existence of a~$2$-equivalent norm $\abs{\,\cdot\,}$ on $X$ which satisfies the desired condition (ii). In fact, it can be taken so that (i) is also valid---{\it cf. }the construction given in the proof of \cite[Thm.~4.2]{gkl}.
\end{proof}

\begin{proposition}\label{p-renorming}
Let $X$ be a separable Banach space with $\Sz(X)=\omega$. Assume that $p(X)<r<q$ and let $B>0$ be any constant so that $\Sz(X,\e)\leq B\e^{-r}$ for every $\e>0$. Then there exist a~number $\gamma=\gamma(B,q,r)>0$ depending only on $B,q$ and $r$, and a~norm $\abs{\,\cdot\,}$ on $X$ satisfying (i) and having the following property:
\begin{itemize}
\item[{\rm (iii)}] if $x^\ast\in X^\ast$ and $(x_n^\ast)_{n=1}^\infty$ is a~weak$^\ast$-null sequence with $\abs{x_n^\ast}\geq\tau$ for some $\tau>0$ and every $n\in\N$, then
$$
\liminf_{n\to\infty}\abs{x^\ast+x_n^\ast}\geq\bigl(\abs{x^\ast}^q+\gamma\tau^q\bigr)^{\! 1/q}.
$$
\end{itemize}
\end{proposition}

\begin{proof}[Proof outline]
According to \cite[Thm.~4.5]{gkl}, there exists an absolute constant $D<10^6$ such that
$$
\Cz(X,\e)\leq\sum_{\substack{k\geq 0\\ 2^k\e/D\leq 1}}\! 2^k\!\cdot\! \Sz(X,2^kD^{-1}\e)\quad\mbox{for every }\e\in (0,1]
$$
and therefore $\Cz(X,\e)\leq A\e^{-r}$ for every $\e\in (0,1]$, where $A=BD^r(1-2^{1-r})^{-1}$. For each $k\in\N$, \cite[Thm.~4.7]{gkl} produces a~$2$-equivalent norm $\abs{\,\cdot\,}_k$ on $X$ (in fact, a~norm which satisfies $\frac{1}{2}\n{x}\leq\abs{x}_k\leq\n{x}$ for $x\in X$) such that for all $x^\ast,x_n^\ast\in X^\ast$ ($n\in\N$) satisfying $\abs{x^\ast}_k=1$, $x_n^\ast\xrightarrow[]{\,w\ast\,}0$ and $\abs{x_n^\ast}_k\geq 2^{-k}$, we have
$$
\liminf_{n\to\infty}\abs{x^\ast+x_n^\ast}_k\geq 1+\frac{1}{\Cz(X,(2^kC)^{-1})}\geq 1+c_1 2^{-kr},
$$
where $c_1=(AC^r)^{-1}$ and $C\leq 19200$ is an absolute constant. Define a~norm $\abs{\,\cdot\,}$ on $X^\ast$ by
$$
\abs{x^\ast}=\frac{1-2^{r-q}}{2^{r-q}}\sum_{k=1}^\infty 2^{k(r-q)}\abs{x^\ast}_k;
$$
it is plainly a~dual norm satisfying $\n{x^\ast}\leq\abs{x^\ast}\leq 2\n{x^\ast}$ for every $x^\ast\in X^\ast$.

Consider arbitrary $x^\ast,x_n^\ast\in X^\ast$ ($n\in\N$) with $\abs{x^\ast}=1$, $x_n^\ast\xrightarrow[]{\,w\ast\,}0$ and $\abs{x_n^\ast}\geq\tau$ for each $n\in\N$ and some $\tau\in (0,1)$. Pick $k\in\N$ satisfying $2^{-k}\leq\frac{1}{4}\tau<2^{1-k}$ and notice that $$\abs{x^\ast}_k\leq 2\n{x^\ast}\leq 2\abs{x^\ast}=2,$$
thus
$$
\abs{x_n^\ast}_k\geq\n{x_n^\ast}\geq\frac{1}{2}\abs{x_n^\ast}\geq\frac{1}{2}\tau\geq 2^{-k}\abs{x^\ast}_k.
$$
Therefore, for that value of $k$ we have
$$
\liminf_{n\to\infty}\abs{x^\ast+x_n^\ast}_k\geq\abs{x^\ast}_k(1+c_1 2^{-kr})
$$
and consequently,
\begin{equation*}
\begin{split}
\liminf_{n\to\infty}\abs{x^\ast+x_n^\ast} &\geq\frac{1-2^{r-q}}{2^{r-q}}\Bigl(\sum_{j\not=k}2^{j(r-q)}\abs{x^\ast}_j+2^{k(r-q)}\abs{x^\ast}_k(1+c_1 2^{-kr})\Bigr)\\
&=1+\frac{1-2^{r-q}}{2^{r-q}}\cdot c_1\abs{x^\ast}_k 2^{-kq}\geq 1+\beta\tau^q,
\end{split}
\end{equation*}
with some $\beta=\beta(B,q,r)$ (we used the fact that $\abs{x^\ast}_k\geq \n{x^\ast}\geq\frac{1}{2}\abs{x^\ast}=\frac{1}{2}$).

Now, consider any non-zero $x^\ast\in X^\ast$ and any weak$^\ast$-null sequence $(x_n^\ast)_{n=1}^\infty$ with $\abs{x_n^\ast}\geq\tau$ for each $n\in\N$. By Bernoulli's inequality and the conclusion of the preceding paragraph, we have
$$
\liminf_{n\to\infty}\abs{x^\ast+x_n^\ast}^q\geq \abs{x^\ast}^q\Bigl(1+\beta\frac{\tau^q}{\abs{x^\ast}^q}\Bigr)^{\!q}\geq \abs{x^\ast}^q\Bigl(1+\beta q\frac{\tau^q}{\abs{x^\ast}^q}\Bigr)=\abs{x^\ast}^q+\beta q\tau^q,
$$
whence the assertion follows with $\gamma=\beta q$.
\end{proof}

\section{$c_0$-sums}
We start with a lemma saying that a~uniform bound for constants of summability of the Szlenk index passes to finite $c_0$-sums.
\begin{lemma}\label{uniform}
Suppose $X_1,\ldots,X_N$ are Banach spaces having summable Szlenk index with the same constant $M$. Then the Banach space $Y=(\bigoplus_{j=1}^N X_j)_{c_0}$ has summable Szlenk index with constant $4M$.
\end{lemma}
\begin{proof}
First, notice that
\begin{equation}\label{L1}
B_{Y^\ast}\subseteq\bigcup_{k_1+\ldots+k_N\leq m+N}\frac{k_1}{m}B_{X_1^\ast}\times\ldots\times \frac{k_N}{m}B_{X_N^\ast}\quad\mbox{for each }m\in\N.
\end{equation}
Indeed, since $Y^\ast=(\bigoplus_{j=1}^N X_j^\ast)_{\ell_1}$, every functional $y^\ast\in B_{Y^\ast}$ has the form $y^\ast=(x_1^\ast,\ldots,x_N^\ast)$ with $x_j^\ast\in X_j^\ast$ and $\sum_{j=1}^N\n{x_j^\ast}\leq 1$. For each $j=1,\ldots,N$ pick the integer $k_j$ so that $(k_j-1)/m<\n{x_j^\ast}\leq k_j/m$. Then each $x_j^\ast\in k_j/m B_{X_j^\ast}$ and 
$$
\frac{1}{m}\Bigl(\sum_{j=1}^Nk_j-N\Bigr)<\sum_{j=1}^N\n{x_j^\ast}=\n{y^\ast}\leq 1,
$$
which yields $\sum_{j=1}^Nk_j\leq m+N$.

Now, for any natural number $r$ denote by $\Theta(r)$ the collection of all $N$-tuples $(q_1,\ldots,q_N)$ satisfying $\sum_{j=1}^Nq_j\geq 1-N/r$, where each $q_j$ is a~non-negative fraction with denominator $r$. We {\it claim} that for all weak$^\ast$-compact sets $K_j\subset X_j^\ast$ ($j=1,\ldots,N$), and any $\e>0$, we have
\begin{equation}\label{L2}
\iota_\e(K_1\times\ldots\times K_N)\subseteq\bigcup_{(q_1,\ldots,q_N)\in\Theta(r)}\iota_{q_1\e/2}K_1\times\ldots\times \iota_{q_N\e/2}K_N\quad\mbox{for each }r\in\N.
\end{equation}
For this, consider any $y^\ast=(x_1^\ast,\ldots,x_N^\ast)\in\iota_\e(K_1\times\ldots\times K_N)$. Then there exists a~net $(y_\alpha^\ast)_{\alpha\in A}\subset K_1\times\ldots\times K_N$, $y_\alpha^\ast=(x_{\alpha,1}^\ast,\ldots,x_{\alpha,N}^\ast)$, which is weak$^\ast$-convergent to $y^\ast$ and satisfies $\n{y_\alpha^\ast-y^\ast}>\e/2$. By passing to a~subnet we may assume that for each $j=1,\ldots,N$ the net $(\n{x_{\alpha,j}^\ast-x_j^\ast})_{\alpha\in A}$ converges to a~non-negative number $\delta_j/2$. For every $j=1,\ldots,N$ pick a~fraction $q_j$ with denominator $r$ so that $q_j\e<\delta_j\leq(q_j+1/r)\e$ (for $\delta_j=0$ we take $q_j=0$ and use the convention $\iota_0K_j=K_j$). Plainly, $x_j^\ast\in\iota_{q_j\e/2}K_j$ for each $j$, and also
$$
\e\leq 2\lim_\alpha\n{y_\alpha^\ast-y^\ast}=2\lim_\alpha\sum_{j=1}^N\n{x_{\alpha,j}^\ast-x_j^\ast}=\sum_{j=1}^N\delta_j\leq\Bigl(\sum_{j=1}^Nq_j+\frac{N}{r}\Bigr)\e
$$
which proves \eqref{L2}.

Our next {\it claim} is the one which allows us to include iterates of the Szlenk derivations of $B_{Y^\ast}$ into products of derivations of $B_{X_j^\ast}$'s. Namely, for all $\e>0$ and $m,r\in\N$ we have
\begin{equation}\label{L3}
\begin{split}
\iota_\e B_{Y^\ast}\subseteq \bigcup_{k_1+\ldots+k_N\leq m+N}\, \bigcup_{(q_1,\ldots,q_N)\in\Theta(r)}\, &\frac{k_1}{m}\iota_{\eta_1}B_{X_1^\ast}\times\ldots\times \frac{k_N}{m}\iota_{\eta_N}B_{X_N^\ast},\\
&\mbox{where }\eta_j=\frac{mq_j\e}{4k_j}\,\mbox{ for }j=1,\ldots,N.
\end{split}
\end{equation}
In order to prove this, note the following two elementary properties of the Szlenk derivation:
\begin{itemize}
\item $\iota_\e(aK)=a\iota_{\e/a}K$;
\item $\iota_\e(\ccup_{i=1}^sK_i)\subseteq\ccup_{i=1}^s\iota_{\e/2}K_i$ (see \cite[Lemma~3.1]{brooker} for a~more general statement),
\end{itemize}
which are valid for all weak$^\ast$-compact sets $K,K_1,\ldots,K_s$ and $a,\e>0$. Combining them with \eqref{L1} and \eqref{L2} we obtain
\begin{equation*}
\begin{split}
\iota_\e B_{Y^\ast} &\subseteq\iota_\e\Biggl(\,\,\bigcup_{k_1+\ldots+k_N\leq m+N}\frac{k_1}{m}B_{X_1^\ast}\times\ldots\times \frac{k_N}{m}B_{X_N^\ast}\,\Biggr)\\
&\subseteq\bigcup_{k_1+\ldots+k_N\leq m+N}\iota_{\e/2}\Biggl(\frac{k_1}{m}B_{X_1^\ast}\times\ldots\times \frac{k_N}{m}B_{X_N^\ast}\Biggr)\\
&\subseteq\bigcup_{k_1+\ldots+k_N\leq m+N}\, \bigcup_{(q_1,\ldots,q_N)\in\Theta(r)}\,\iota_{q_1\e/4}\frac{k_1}{m}B_{X_1^\ast}\times\ldots\times \iota_{q_N\e/4}\frac{k_N}{m}B_{X_N^\ast}\\
&=\bigcup_{k_1+\ldots+k_N\leq m+N}\, \bigcup_{(q_1,\ldots,q_N)\in\Theta(r)}\, \frac{k_1}{m}\iota_{\eta_1}B_{X_1^\ast}\times\ldots\times \frac{k_N}{m}\iota_{\eta_N}B_{X_N^\ast},
\end{split}
\end{equation*}
as desired.

Finally, fix $\e_1,\ldots,\e_n>0$ so that $\iota_{\e_1}\ldots\iota_{\e_n}B_{Y^\ast}\not=\varnothing$, and pick arbitrary $m,r\in\N$. In view of \eqref{L3}, we have
\begin{equation*}
\begin{split}
\iota_{\e_1}\iota_{\e_2}B_{Y^\ast} &\subseteq \iota_{\e_1}\Biggl(\,\,\bigcup_{k_1+\ldots+k_N\leq m+N}\, \bigcup_{(q_{2,1},\ldots,q_{2,N})\in\Theta(r)} \frac{k_1}{m}\iota_{\eta_{2,1}}B_{X_1^\ast}\times\ldots\times \frac{k_N}{m}\iota_{\eta_{2,N}}B_{X_N^\ast}\,\Biggr)\\
&\subseteq\bigcup_{k_1+\ldots+k_N\leq m+N}\, \bigcup_{(q_{2,1},\ldots,q_{2,N})\in\Theta(r)}\,\iota_{\e_1/2}\Biggl(\frac{k_1}{m}\iota_{\eta_{2,1}}B_{X_1^\ast}\times\ldots\times \frac{k_N}{m}\iota_{\eta_{2,N}}B_{X_N^\ast}\Biggr)\\
&\subseteq\bigcup_{k_1+\ldots+k_N\leq m+N}\, \bigcup_{\substack{(q_{1,1},\ldots,q_{1,N})\in\Theta(r)\\ (q_{2,1},\ldots, q_{2,N})\in\Theta(r)}}\, \frac{k_1}{m}\iota_{\eta_{1,1}}\iota_{\eta_{2,1}}B_{X_1^\ast}\times\ldots\times \frac{k_N}{m}\iota_{\eta_{1,N}}\iota_{\eta_{2,N}}B_{X_N^\ast},
\end{split}
\end{equation*}
where $\eta_{i,j}=mq_{i,j}\e_i/4k_j$ for $i=1,2$ and $j=1,\ldots,N$, whence by induction we obtain
$$
\iota_{\e_1}\ldots\iota_{\e_n}B_{Y^\ast}\subseteq\bigcup_{k_1+\ldots+k_N\leq m+N}\, \bigcup_{\substack{(q_{1,1},\ldots,q_{1,N})\in\Theta(r)\\ \vdots\\ (q_{n,1},\ldots,q_{n,N})\in\Theta(r)}}\, \frac{k_1}{m}\iota_{\eta_{1,1}}\ldots\iota_{\eta_{n,1}}B_{X_1^\ast}\times\ldots\times \frac{k_N}{m}\iota_{\eta_{1,N}}\ldots\iota_{\eta_{n,N}}B_{X_N^\ast}
$$
with $\eta_{i,j}=mq_{i,j}\e_i/4k_j$ for $i=1,\ldots,n$ and $j=1,\ldots,N$.

Since $\iota_{\e_1}\ldots\iota_{\e_n}B_{Y^\ast}\not=\varnothing$, one of the summands above must be non-empty, and hence we have $\iota_{\eta_{1,j}}\ldots\iota_{\eta_{n,j}}B_{X_j^\ast}\not=\varnothing$ for each $j=1,\ldots,N$ which by the assumption implies that $\sum_{i=1}^n\eta_{i,j}\leq M$, i.e.
$$
\sum_{i=1}^nq_{i,j}\e_i\leq\frac{4k_jM}{m}\quad\mbox{for each }j=1,\ldots,N.
$$
Adding these inequalities over all $j$'s yields
$$
\sum_{i=1}^n\Bigl(1-\frac{N}{r}\Bigr)\e_i\leq\sum_{i=1}^n\sum_{j=1}^Nq_{i,j}\e_i\leq 4M\sum_{j=1}^N\frac{k_j}{m}\leq 4M\cdot\frac{m+N}{m},
$$
thus by letting $m,r\to\infty$ we obtain $\sum_{i=1}^n\e_i\leq 4M$.
\end{proof}

We now proceed to our first main result. The following standard notation will be used: if $X$ is a~direct sum of $X_n$'s, then for any interval $I\subseteq\N$, $P_I$ stands for canonical projection onto the corresponding direct sum of $X_n$'s with $n\in I$; for any $j\in\N$, $P_j$ stands for $P_{\{j\}}$. If a~given direct sum decomposition is a~dual decomposition, then $P_I$'s are adjoints to the corresponding canonical inclusion operators, and hence they are weak$^\ast$-to-weak$^\ast$ continuous. We will be using this fact without mentioning.

\begin{theorem}\label{c0-sum}
For any sequence $(X_n)_{n=1}^\infty$ of separable Banach spaces with uniformly summable Szlenk index, the Banach space $X=(\bigoplus_{n=1}^\infty X_n)_{c_0}$ has summable Szlenk index.
\end{theorem}
\begin{proof}
Let $M$ be a~constant of summability of the Szlenk index for all the spaces $X_n$ ($n\in\N$) and let $c=c(4M)$ be as in Lemma~\ref{renorming}. Fix arbitrarily small $\eta>0$.

Assume $\e_1,\ldots,\e_n>0$ are such that $\iota_{\e_1}\ldots\iota_{\e_n}B_{X^\ast}\not=\varnothing$. Then there exists a~weak$^\ast$-null tree-map $(x_a^\ast)_{a\in S}$ in $X^\ast$ such that $\n{x_a^\ast}\geq\frac{1}{4}\e_{\abs{a}}$ for each $a\in S$ with $1\leq\abs{a}\leq n$, and $\n{\sum_{a\in\beta}x_a^\ast}\leq 1$ for every branch $\beta\subset S$. We are going to define inductively a~sequence $(x_{(\nu_1,\ldots,\nu_k)}^\ast\colon k=0,1,\ldots,n)$ assigned to $S$ along the partial branch determined by a~certain node $(\nu_1,\ldots,\nu_n)\in S$.

Set $\nu_1=1$ and choose $N_1\geq 1$ so large that $\n{P_{(N_1,\infty)}x_{(\nu_1)}^\ast}<\eta$. Define 
$$
Z_1=\Bigl(\bigoplus_{i=1}^{N_1}X_i\Bigr)_{c_0}\,\,\,\mbox{and }z_1^\ast=P_{[1,N_1]}x_{(\nu_1)}^\ast\in Z_1^\ast.
$$
By Lemma~\ref{uniform}, the space $Z_1$ has summable Szlenk index with constant $4M$ and therefore we may apply Lemma~\ref{renorming} to $Z_1$ and
\begin{equation}\label{tau}
\tau\coloneqq\frac{\eta}{2n(1+\eta)}.
\end{equation}
We obtain a~norm $\abs{\,\cdot\,}_1$ on $Z_1$ such that its dual norm satisfies $\n{x^\ast}\leq\abs{x^\ast}_1\leq 2\n{x^\ast}$ for every $x^\ast\in Z^\ast$, as well as condition (ii) with $Z_1$ instead of $Z$ and $\abs{\,\cdot\,}_1$ instead of $\abs{\,\cdot\,}$.

Let $1\leq k<n$ and suppose we have already chosen integers $0\eqqcolon N_0<N_1<\ldots<N_k$ and a~partial branch $((\nu_1,\ldots,\nu_i)\colon i\leq k)$ of $S$. Define
$$
Z_j=\Bigl(\bigoplus_{i=N_{j-1}+1}^{N_j}X_i\Bigr)_{c_0},\,\mbox{ so that }\,\,Z_j^\ast=\Bigl(\bigoplus_{i=N_{j-1}+1}^{N_j}X_i^\ast\Bigr)_{\ell_1}\quad\mbox{for }j=1,\ldots,k
$$
and suppose we have also picked certain norms $\abs{\,\cdot\,}_1,\ldots,\abs{\,\cdot\,}_k$ on $Z_1,\ldots,Z_k$, respectively. For each $1\leq j\leq k$ define $\abs{\,\cdot\,}_{1,\ldots,j}$ as the norm on the $\ell_1$-direct sum $((Z_1^\ast,\abs{\,\cdot\,}_1)\oplus\ldots\oplus (Z_j^\ast,\abs{\,\cdot\,}_j))_{\ell_1}$. (We use the same notation for dual norms as for the corresponding predual ones.) Define $z_1^\ast,\ldots,z_k^\ast$ by the formula 
$$
z_j^\ast=z_{j-1}^\ast+P_{[1,N_j]}x_{(\nu_1,\ldots,\nu_j)}^\ast\quad\mbox{for }j=2,\ldots,k.
$$
Our induction hypothesis says that for each $j=1,\ldots,k$ the following clues are satisfied:
\begin{itemize}
\item[(h1)] $\n{P_{(N_j,\infty)}x_{(\nu_1,\ldots,\nu_j)}^\ast}<\eta$;

\vspace*{1mm}
\item[(h2)] $\frac{1}{2}\n{x}\leq\abs{x}_j\leq\n{x}$ for every $x\in Z_j$, so that $\n{x^\ast}\leq\abs{x^\ast}_j\leq 2\n{x^\ast}$ for every $x^\ast\in Z_j^\ast$;

\vspace*{1mm}
\item[(h3)]  if $x^\ast\in Z_j^\ast$, $\abs{x^\ast}_j=1$ and $(x_m^\ast)_{m=1}^\infty\subset Z_j^\ast$ is a~weak$^\ast$-null sequence with $\abs{x_m^\ast}_j\geq\xi\geq\tau$ for $m\in\N$, then $\liminf_{m\to\infty}\abs{x^\ast+x_m^\ast}_j\geq 1+c\xi$;

\vspace*{1mm}
\item[(h4)] $\displaystyle{\abs{z_j^\ast}_{1,\ldots,j}\geq \frac{c}{4(1+c)}(\e_1+\ldots+\e_j)-j\eta}$.
\end{itemize}
Now, we shall consider two possibilities.

\vspace*{2mm}\noindent
{\bf (p$^\prime$):} $\n{P_{(N_k,\infty)}x_{(\nu_1,\ldots,\nu_k,\nu)}^\ast}>\frac{c}{4(1+c)}\e_{k+1}$ for infinitely many $\nu$'s. Since
$$
P_{[1,N_k]}x_{(\nu_1,\ldots,\nu_k,\nu)}^\ast\xrightarrow[\,\,\nu\to\infty\,\,]{w^\ast}0\,\,\mbox{ in }(Z_1^\ast\oplus\ldots\oplus Z_k^\ast)_{\ell_1},
$$
we can then choose $\nu\in\N$ to guarantee that 
$$
\bigl|z_k^\ast+P_{[1,N_k]}x_{(\nu_1,\ldots,\nu_k,\nu)}^\ast\bigr|_{1,\ldots,k}>\abs{z_k^\ast}_{1,\ldots,k}-\eta,
$$
and then arrange also $N_{k+1}>N_k$ so that
$$
\bigl\|P_{(N_k,N_{k+1}]}x_{(\nu_1,\ldots,\nu_k,\nu)}^\ast\bigr\|>\frac{c}{4(1+c)}\e_{k+1}\,\,\mbox{ and }\,\,\bigl\|P_{(N_{k+1},\infty)}x_{(\nu_1,\ldots,\nu_k,\nu)}^\ast\bigr\|<\eta.
$$
Choose any norm $\abs{\,\cdot\,}_{k+1}$ on the space $Z_{k+1}=(\bigoplus_{i=N_k+1}^{N_{k+1}}X_i)_{c_0}$ according to Lemma~\ref{renorming} (for the same fixed parameter $\tau$ as earlier). Define also $\abs{\,\cdot\,}_{1,\ldots,k+1}$ and $z_{k+1}^\ast$ as previously, taking $j=k+1$. Then our induction hypothesis (h4) gives
\begin{equation*}
\begin{split}
\abs{z_{k+1}^\ast}_{1,\ldots,k+1} &>\abs{z_k^\ast}_{1,\ldots,k}+\frac{c}{4(1+c)}\e_{k+1}-\eta\\
&\geq \frac{c}{4(1+c)}(\e_1+\ldots+\e_{k+1})-(k+1)\eta,
\end{split}
\end{equation*}
and hence (h4) is valid for $j=k+1$. Plainly, the same is true for (h1)--(h3).

\vspace*{2mm}\noindent
{\bf (p$^{\prime\prime}$):} (p$^\prime$) is false. In this case for all $\nu$'s from some infinite set $\mathcal{N}\subseteq\N$ we have $\n{P_{(N_k,\infty)}x_{(\nu_1,\ldots,\nu_k,\nu)}^\ast}\leq \frac{c}{4(1+c)}\e_{k+1}$ and since $\n{x_a^\ast}\geq\frac{1}{4}\e_{k+1}$ for every $a\in S$ with $\abs{a}=k+1$, we must have
\begin{equation}\label{P1N}
\bigl|P_{[1,N_k]}x_{(\nu_1,\ldots,\nu_k,\nu)}^\ast\bigr|_{1,\ldots,k}\geq\frac{\e_{k+1}}{4(1+c)}\quad\mbox{for each }\nu\in\mathcal{N}.
\end{equation}
Now, we shall exploit properties of the norms $\abs{\,\cdot\,}_1,\ldots\abs{\,\cdot\,}_k$. Recall that all the values of our tree-map have norms at most $n$. Since $z_k^\ast$ is one of those values after executing $k$ cuts, each resulting in decrease of norm by at most $\eta$, we have $\n{z_k^\ast}\leq n(1+\eta)$, thus $\abs{z_k^\ast}_{1,\ldots,k}\leq 2n(1+\eta)$. For each $j=1,\ldots,k$ set $I_j=(N_{j-1},N_j]$ and observe that on the one hand we have
\begin{equation}\label{inf1}
\liminf_{\nu\in\mathcal{N}}\,\bigl|P_{I_j}(z_k^\ast+x_{(\nu_1,\ldots,\nu_k,\nu)}^\ast)\bigr|_j\geq\bigl|P_{I_j}z_k^\ast\bigr|_j\quad\mbox{(because }P_{I_j}x_{(\nu_1,\ldots,\nu_k,\nu)}^\ast\xrightarrow[\,\nu\to\infty\,]{w^\ast}0),
\end{equation}
and on the other, property (h3) implies that for every infinite set $\mathcal{M}\subseteq\mathcal{N}$ we have
\begin{equation}\label{inf2}
\begin{split}
\liminf_{\nu\in\mathcal{M}}\,\bigl|P_{I_j}(z_k^\ast+x_{(\nu_1,\ldots,\nu_k,\nu)}^\ast)\bigr|_j &=\bigl|P_{I_j}z_k^\ast\bigr|_j\!\cdot\liminf_{\nu\in\mathcal{M}}\,\Biggl|\frac{P_{I_j}z_k^\ast}{\bigl|P_{I_j}z_k^\ast\bigr|_j}+\frac{P_{I_j}x_{(\nu_1,\ldots,\nu_k,\nu)}^\ast}{\bigl|P_{I_j}z_k^\ast\bigr|_j}\Biggr|_j\\
& \geq \bigl|P_{I_j}z_k^\ast\bigr|_j+c\!\cdot\!\liminf_{\nu\in\mathcal{M}}\, \bigl|P_{I_j}x_{(\nu_1,\ldots,\nu_k,\nu)}^\ast\bigr|_j,
\end{split}
\end{equation}
provided that $\liminf_{\nu\in\mathcal{M}}\,\vert P_{I_j}x_{(\nu_1,\ldots,\nu_k,\nu)}^\ast\vert_j\geq\tau\abs{P_{I_j}z_k^\ast}_j$. Fix, for a~moment, any infinite set $\mathcal{M}\subseteq\mathcal{N}$ and define
$$
a_j=\liminf_{\nu\in\mathcal{M}}\,\bigl|P_{I_j}x_{(\nu_1,\ldots,\nu_k,\nu)}^\ast\bigr|_j\,\,\mbox{ and }\,\, b_j=\bigl|P_{I_j}z_k^\ast\bigr|_j\quad\mbox{for }j=1,\ldots,k.
$$
Then, in view of \eqref{inf1} and \eqref{inf2}, we have
\begin{equation*}
\begin{split}
\liminf_{\nu\in\mathcal{M}}\,\bigl|z_k^\ast+P_{[1,N_k]} &x_{(\nu_1,\ldots,\nu_k,\nu)}^\ast\bigr|_{1,\ldots,k}\geq \sum_{j=1}^k\liminf_{\nu\in\mathcal{M}}\,\bigl|P_{I_j}(z_k^\ast+x_{(\nu_1,\ldots,\nu_k,\nu)}^\ast)\bigr|_j\\
&\geq \sum_{\{j\colon a_j<\tau b_j\}}b_j+\sum_{\{j\colon a_j\geq\tau b_j\}}\!\!(b_j+ca_j)=\abs{z_k^\ast}_{1,\ldots,k}+c\!\!\!\sum_{\{j\colon a_j\geq\tau b_j\}}a_j.
\end{split}
\end{equation*}
Observe that 
\begin{equation*}
\begin{split}
\sum_{\{j\colon a_j\geq\tau b_j\}}\!\!\! a_j &=\sum_{j=1}^ka_j-\!\!\!\!\sum_{\{j\colon a_j<\tau b_j\}}\!\!\! a_j\geq\sum_{j=1}^ka_j-\tau\!\!\!\!\sum_{\{j\colon a_j<\tau b_j\}}b_j\\
&\geq\sum_{j=1}^ka_j-\tau\sum_{j=1}^kb_j\geq\sum_{j=1}^ka_j-2n\tau(1+\eta),
\end{split}
\end{equation*}
therefore by picking $\mathcal{M}$ so that all the sequences $(\vert P_{I_j}x_{(\nu_1,\ldots,\nu_k,\nu)}^\ast\vert_j)_{\nu\in\mathcal{M}}$ converge and using \eqref{P1N} we arrive at
\begin{equation*}
\begin{split}
\liminf_{\nu\in\mathcal{M}}\,\bigl|z_k^\ast &+P_{[1,N_k]}x_{(\nu_1,\ldots,\nu_k,\nu)}^\ast\bigr|_{1,\ldots,k} \geq \bigl|z_k^\ast\bigr|_{1,\ldots,k}+c\Bigl(\sum_{j=1}^ka_j-2n\tau(1+\eta)\Bigr)\\
&=\bigl|z_k^\ast\bigr|_{1,\ldots,k}+c\cdot\lim_{\nu\in\mathcal{M}}\bigl|P_{[1,N_k]}x_{(\nu_1,\ldots,\nu_k,\nu)}^\ast\bigr|_{1,\ldots,k}-2nc\tau(1+\eta)\\
&\geq \bigl|z_k^\ast\bigr|_{1,\ldots,k}+\frac{c\e_{k+1}}{4(1+c)}-2nc\tau(1+\eta)>\bigl|z_k^\ast\bigr|_{1,\ldots,k}+\frac{c\e_{k+1}}{4(1+c)}-\eta;
\end{split}
\end{equation*}
the last inequality follows from \eqref{tau} and the fact that $c<1$. Therefore, by the induction hypothesis (h4), there exists $\nu_{k+1}\in\N$ for which
$$
\bigl|z_k^\ast+P_{[1,N_k]}x_{(\nu_1,\ldots,\nu_{k+1})}^\ast\bigr|_{1,\ldots,k}\geq \frac{c}{4(1+c)}\sum_{j=1}^{k+1}\e_j-(k+1)\eta.
$$

Now, pick $N_{k+1}>N_k$ so that (h1) holds true for $j=k+1$, define $Z_{k+1}=(\bigoplus_{i=N_k+1}^{N_{k+1}}X_i)_{c_0}$ and use Lemma~\ref{renorming} to produce a~norm $\abs{\,\cdot\,}_{k+1}$ on $Z_{k+1}$ fulfilling conditions (h2) and (h3) for $j=k+1$. The last estimate shows that (h4) is also satisfied with $j=k+1$ and
$$
z_{k+1}^\ast=z_k^\ast+P_{[1,N_{k+1}]}x_{(\nu_1,\ldots,\nu_{k+1})}^\ast.
$$
This finishes the inductive construction.

Having defined the partial branch $\beta=((\nu_1,\ldots,\nu_i)\colon i\leq n)$, notice that
\begin{equation*}
\begin{split}
1 &\geq \Biggl\|\sum_{a\in\beta}x_a^\ast\Biggr\|\geq\Biggl\|\sum_{k=1}^nP_{[1,N_k]}x_{(\nu_1,\ldots,\nu_k)}^\ast\Biggr\|-n\eta\\
&\geq\frac{1}{2}\bigl|z_n^\ast\bigr|_{1,\ldots,n}\!\!-n\eta\geq\frac{c}{8(1+c)}\sum_{j=1}^n\e_j-\frac{3}{2}n\eta,
\end{split}
\end{equation*}
which completes the proof.
\end{proof}
As a consequence, we infer that the $c_0$-sum of countably many copies of the original Tsirelson space, $c_0(\TT)$, has summable Szlenk index. Note that this does not follow directly from \cite[Prop.~6.7]{kos} (see Theorem~\ref{kos_theorem} below), since the natural basis of $(c_0(\TT))^\ast$ does not admit any blocking with an~asymptotic $\ell_1$ structure.  

\section{More general direct sums}
Although $c_0$ is by far the most prototypical example of a Banach space with summable Szlenk index, there are also other interesting ones having unconditional basis. Therefore, it is natural to ask whether Theorem~\ref{c0-sum} remains valid for more general sums than $c_0$-sums. Before revealing the fact that, unfortunately, it is not the case, let us recall some basic facts concerning infinite direct sums of Banach spaces.

Let $\EE$ be a Banach space with a~normalized, $1$-unconditional basis $(e_n)_{n=1}^\infty$ and $(X_n)_{n=1}^\infty$ be a~sequence of Banach spaces. The $\EE$-direct sum of $(X_n)_{n=1}^\infty$, denoted $(\bigoplus_{n=1}^\infty X_n)_\EE$, is the Banach space consisting of all sequences $(x_n)_{n=1}^\infty\in\prod_{n=1}^\infty X_n$ for which the series $\sum_{n=1}^\infty\n{x_n}e_n$ converges, equipped with the coordinatewise operations and the norm
$$
\n{(x_n)_{n=1}^\infty}=\Biggl\|\sum_{n=1}^\infty\n{x_n}e_n\Biggr\|.
$$

Let $(e_n^\ast)_{n=1}^\infty\subset\EE^\ast$ be the sequence of biorthogonal functionals corresponding to the basis $(e_n)_{n=1}^\infty$ and let $\FF=\oo{\mathrm{span}}\{e_n^\ast\colon n\in\N\}$. Duality is described with the aid of an operator 
$$
\Upsilon\colon \Bigl(\bigoplus_{n=1}^\infty X_n^\ast\Bigr)_{\!\!\FF}\xrightarrow[]{\quad} \Bigl(\bigoplus_{n=1}^\infty X_n\Bigr)_{\!\!\EE}^{\!\!\ast},\quad \langle\boldsymbol{x},\Upsilon(\boldsymbol{x^\ast})\rangle=\sum_{n=1}^\infty\langle x_n,x_n^\ast\rangle\,\mbox{ for }\boldsymbol{x}=(x_n)_{n=1}^\infty,\,\boldsymbol{x^\ast}=(x_n^\ast)_{n=1}^\infty.
$$
This definition makes sense as 
$$
\sum_{n=1}^\infty\abs{\langle x_n,x_n^\ast\rangle}\leq \sum_{n=1}^\infty\n{x_n}\n{x_n^\ast}=\Biggl|\Biggl\langle\sum_{n=1}^\infty\n{x_n}e_n,\sum_{n=1}^\infty\n{x_n^\ast}e_n^\ast\Biggr\rangle\Biggr|\leq\n{\boldsymbol{x}}\n{\boldsymbol{x^\ast}}.
$$
According to \cite[Lemma~4.2]{laustsen} the operator $\Upsilon$ is always an isometry. Moreover, we have the following list of equivalent conditions saying when $\Upsilon$ is actually an~isometric isomorphisms (see \cite[Prop.~4.8]{laustsen}):
\begin{itemize*}
\item[(a)] $\Upsilon$ is surjective for any sequence $(X_n)_{n=1}^\infty$ of Banach spaces;
\item[(b)] the basis $(e_n)_{n=1}^\infty$ is shrinking, i.e. for every $\xi\in\EE^\ast$ we have $\n{\xi\vert_{\oo{\mathrm{span}}\{e_n\colon n\geq k\}}}\xrightarrow[k\to\infty]{}0$;

\vspace*{-1mm}\item[(c)] $\EE^\ast=\FF$.
\end{itemize*}

\begin{example}\label{T(c0)}
Let $\TT$ be Tsirelson's space, so that its dual $\TT^\ast$ is the space $T$ considered in \cite{figiel_johnson}, that is, the completion of $c_{00}$ with respect to the norm $\n{\cdot}$ satisfying the implicit formula
$$
\n{x}=\n{x}_{c_0}\,\vee\,\frac{1}{2}\sup\Biggl\{\sum_{j=1}^k\n{E_jx}\colon \{k\}<E_1<\ldots<E_k\Biggr\}\quad\mbox{for }x\in c_{00}
$$
(with the usual notation $Ex=\ind_E\!\cdot\! x$ and $E<F$ if $\max E<\min F$ or either of these sets is empty). Despite the fact that $\TT$ has summable Szlenk index (see \cite[Prop.~6.7]{kos}), the $\TT$-direct sum $\TT(c_0)$ of infinitely many copies of $c_0$ does not. Indeed, since $\TT^\ast$ is not isomorphic to $\ell_1$, there is a~sequence $(\boldsymbol{t}_k)_{k=1}^\infty\subset\TT^\ast$,
$$
\boldsymbol{t}_k=(t_{k,1},\ldots,t_{k,n_k},0,0,\ldots)\,\,\,\mbox{ with }t_{k,i}>0\,\mbox{ for each }i=1,\ldots,n_k,
$$
such that $\n{\boldsymbol{t}_k}=1$ and $\sum_{i=1}^{n_k}t_{k,i}\to\infty$ as $k\to\infty$. For a~fixed $k\in\N$ and each $j=1,\ldots,n_k$ we choose any weak$^\ast$-null sequence $(\xi_{j,m})_{m=1}^\infty\subset\ell_1$ so that $\n{\xi_{j,m}}=t_{k,j}$. We construct a~weak$^\ast$-null tree-map $(x_a^\ast)_{a\in S}$ of height $n_k+1$ by requiring that for every $a\in S$ with $\abs{a}<n_k$ we have 
$$
x_{a\smallfrown m}^\ast=\bigl(\underbrace{0,\ldots,0}_{\abs{a}\text{ times}},\xi_{\abs{a}+1,m},0,0,\ldots\bigr)\quad\mbox{for }m>\max a.
$$
Then observe that for every branch $\beta\subset S$ there are indices $m_1,\ldots,m_{n_k}$ such that
$$
\sum_{a\in\beta}x_a^\ast=\bigl(\xi_{1,m_1},\ldots,\xi_{n_k,m_{n_k}},0,0,\ldots\bigr)
$$
and hence $\n{\sum_{a\in\beta}x_a^\ast}\leq 1$. We have shown that  $\iota_{t_{k,1}}\ldots\iota_{t_{k,n_k}}B_{\TT(c_0)^\ast}\not=\varnothing$ for every $k\in\N$
and since $\sum_{i=1}^{n_k}t_{k,i}\to\infty$, we conclude that $\TT(c_0)$ fails to have summable Szlenk index. 
\end{example}

The success of Example~\ref{T(c0)} was based on the fact that already the partial finite direct sums had `constants of summability' growing to infinity. This suggests that summability of Szlenk index may be preserved whenever we can avoid that particular~situation, and this is indeed the sense of our next result. It can be also derived from a~result due to Knaust, Odell and Schlumprecht \cite{kos} (see Theorem~\ref{kos_theorem} below) by using the fact that every unconditional basis of a~space with separable dual must be shrinking ({\it cf. }\cite[Theorem~3.3.1]{ak}), and hence it gives rise to a~natural finite-dimensional decomposition of the space $X$ below. Here, we present a~slightly different approach.
\begin{lemma}[{\it cf. }{\cite[Lemma~1.c.11]{lt}}]\label{weak_conv}
Let $(e_k;e_k^\ast)_{k=1}^\infty$ be an~unconditional basis of a~Banach space $X$. Suppose that $(x_n)_{n=1}^\infty$ is a~bounded sequence in $X$ such that for every $x^\ast\in X^\ast$ the limit $\lim_{n\to\infty}x^\ast(x_n)$ exists and $\lim_{n\to\infty}e_k^\ast(x_n)=0$ for each $k\in\N$. Then $x_n\xrightarrow[]{\,\,w\,\,}0$.
\end{lemma}
\begin{theorem}\label{E-sum}
Let $\EE$ be a Banach space with an unconditional basis and summable Szlenk index. Then for every sequence $(X_n)_{n=1}^\infty$ of finite-dimensional Banach spaces the space $X=(\bigoplus_{n=1}^\infty X_n)_\EE$ has summable Szlenk index.
\end{theorem}
\begin{proof}
Let $(e_n)_{n=1}^\infty$ be a~normalized, $1$-unconditional basis of $\EE$ and let $(e_n^\ast)_{n=1}^\infty$ be the corresponding sequence of biorthogonal functionals. We consider two~quantities connected with weakly null tree-maps: $N(\sigma)$ (already defined in Section~2) and its formally `weaker' version $\w{N}(\sigma)$. Namely, for any $\sigma\in (0,1)$ the number $N(\sigma)\in\N\cup\{\infty\}$ ($\w{N}(\sigma)$, respectively) is the least natural number $N$ for which there exists a~tree-map $(x_a)_{a\in S}$ in $\EE$ of height $N+1$ such that:
\begin{itemize}
\item $x_{a\smallfrown n}\xrightarrow[]{\,\,w\,\,}0$ for every $a\in S$ (respectively: $e_k^\ast(x_{a\smallfrown n})\xrightarrow[]{\quad}0$ for all $a\in S$ and $k\in\N$);
\item $\n{x_a}\leq\sigma$ for every $a\in S$;
\item $\bigl\|\sum_{a\in\beta}x_a\bigr\|>1$ for every branch $\beta\subset S$,
\end{itemize}
or is equal to $\infty$ if no such $N$ exists.

In fact, we have $N(\sigma)=\w{N}(\sigma)$ for every $\sigma\in (0,1)$. Indeed, the inequality ${\w{N}(\sigma)\leq N(\sigma)}$ is obvious; for the converse, assume that we are given a~tree-map $(x_a)_{a\in S}$ in $X$ with the properties listed in the definition of $\w{N}(\sigma)$. Since $\EE^\ast$ is separable, to every sequence of the form $(x_{a\smallfrown n})_{n>\max a}$ we can apply the standard diagonal procedure in order to obtain a~subsequence $(x_{a\smallfrown n_k})_{k=1}^\infty$ such that the limit $\lim_{k\to\infty} x^\ast(x_{a\smallfrown n_k})$ exists for every $x^\ast\in\EE^\ast$. According to Lemma~\ref{weak_conv} we have then $x_{a\smallfrown n_k}\!\!\!\xrightarrow[]{\,\,w\,\,}0$. By an obvious inductive procedure, going in the direction of increasing heights, we obtain a~full subtree $T\subseteq S$ such that $(x_a)_{a\in T}$ is a~weakly null tree-map in $\EE$. Hence $\w{N}(\sigma)\geq N(\sigma)$.

Assume that $X$ does not have summable Szlenk index; then \cite[Theorem~4.10]{gkl} implies that for every $\sigma\in (0,1)$ there is a~weakly null tree-map $(\boldsymbol{x}_a)_{a\in S}$ in $X$ such that $\n{\boldsymbol{x}_a}\leq\sigma$ for $a\in S$ and $\n{\sum_{a\in\beta}\boldsymbol{x}_a}>1$ for every branch $\beta\subset S$. Now, consider a~tree-map $(y_a)_{a\in S}$ in $\EE$ given by
$$
y_a=\sum_{k=1}^\infty\n{x_k^{(a)}}e_k ,\quad\mbox{where }\boldsymbol{x}_a=(x_k^{(a)})_{k=1}^\infty.
$$
Of course, $\n{y_a}=\n{\boldsymbol{x}_a}\leq\sigma$ for each $a\in S$. Also, for every branch $\beta\subset S$ we have
$$
\Biggl\|\sum_{a\in\beta}y_a\Biggr\|=\Biggl\|\sum_{a\in\beta}\sum_{k=1}^\infty\bigl\|x_k^{(a)}\bigr\|e_k\Biggr\|\geq\Biggl\|\sum_{k=1}^\infty\Biggl\|\sum_{a\in\beta}x_k^{(a)}\Biggr\|e_k\Biggr\|=\Biggl\|\sum_{a\in\beta}\boldsymbol{x}_a\Biggr\|>1.
$$
Moreover, since $X_k$ is finite-dimensional, we have $\lim_{n\to\infty}\n{x_k^{(a\smallfrown n)}}=0$ for every $k\in\N$ and $a\in S$. This means that $\w{N}(\sigma)<\infty$ and hence also $N(\sigma)<\infty$ for all $\sigma$'s which, in view of \cite[Theorem~4.10]{gkl}, is impossible as $\EE$ has summable Szlenk index.
\end{proof}
\begin{remark}
Note that the above result is, in a~sense, `tautological' in the case where $\EE=c_0$. Indeed, since every finite-dimensional space embeds $(1+\e)$-isometrically in $c_0$, for any $\e>0$ ({\it cf. }\cite[\S 11.1]{ak}), the space $(\bigoplus_{n=1}^\infty X_n)_{c_0}$ is `almost isometric' to a~subspace of $c_0$, so it has summable Szlenk index for that reason.
\end{remark}

\section{Summability and power type}
Recall that the Szlenk power type $p(X)$ is defined as the infimum over those $q\geq 1$ which admit an estimate $\Sz(X,\e)\leq C\e^{-q}$ for every $\e\in (0,1)$ and some $C>0$. In this section, we are interested in how the Szlenk power type behaves with respect to reasonable direct sums. In order to obtain some relevant result, it is not enough just to assume that the Szlenk power types of all summands are bounded, because the corresponding constants $C$, if too large, may make the Szlenk index of the direct sum even larger than $\omega$. 
\begin{example}
For each $N\in\N$ the space $C([0,\omega^N])$ is isomorphic to $c_0$, whence its Szlenk power type equals~$1$. However, $(\bigoplus_{N=1}^\infty C([0,\omega^N]))_{c_0}$ is isomorphic to $C([0,\omega^\omega])$ which is known to have Szlenk index $\omega^2$.
\end{example}

This drastic example motivates the following definition. For any Banach space $X$ with $\Sz(X)\leq\omega$ we set 
$$
C_p(X)=\inf\bigl\{c>0\colon \Sz(X,\e)\leq c\e^{-p}\mbox{ for every }\e\in (0,1)\bigr\}.
$$
We say that a sequence $(X_n)_{n=1}^\infty$ of Banach spaces with $\Sz(X_n)\leq\omega$ ($n\in\N$) is {\it power type bounded}, provided that the number
$$
\mathfrak{p}(X_n)_{n=1}^\infty\coloneqq\inf\Bigl\{p\in [1,\infty)\colon \sup_n\;\! C_{p}(X_n)<\infty\Bigr\}
$$
is finite. This is what we assume about the sequence $(X_n)_{n=1}^\infty$ in the direct sum $(\bigoplus_{n=1}^\infty X_n)_\EE$ considered. Next, we shall focus on the underlying space $\EE$.

Recall that a sequence $(E_n)_{n=1}^\infty$ of finite-dimensional spaces is called a~{\it finite-dimensional decomposition} (FDD) of a~Banach space $X$ if every $x\in X$ may be written uniquely as $x=\sum_{n=1}^\infty x_n$, where each $x_n\in X_n$. We use the standard notation  $P_n$ and $P_I$, with $n\in\N$ and $I\subset\N$ being an interval, for projections corresponding to a~given FDD. Following \cite{MT-J} we shall say that $X$ is $C$-{\it asymptotic} $\ell_p$ (with some $C\geq 1$ and $1\leq p\leq\infty$) {\it with respect to} its FDD $(E_n)_{n=1}^\infty$ if for every block sequence $(x_j)_{j=1}^n$ of $(E_j)_{j=n}^\infty$ we have
\begin{equation}\label{asymp_def}
\frac{1}{C}\Biggl(\sum_{j=1}^n\n{x_j}^p\Biggr)^{\!\!\! 1/p}\leq\Biggl\|\sum_{j=1}^nx_j\Biggr\|\leq C\Biggl(\sum_{j=1}^n\n{x_j}^p\Biggr)^{\!\!\!1/p}
\end{equation}
(if $p=\infty$, we take the $c_0$-norm $\max_{1\leq j\leq n}\n{x_j}$). We say that $X$ is {\it asymptotic} $\ell_p$ {\it with respect to} $(E_n)_{n=1}^\infty$ if it is $C$-asymptotic $\ell_p$ with respect to $(E_n)_{n=1}^\infty$ for some $C\geq 1$. 

A~more general approach and a~coordinate free way of stating the definition of asymptotic $\ell_p$ spaces was given in \cite{MMT-J} with the aid of certain infinite games. In view of results by Odell, Schlumprecht and Zs\'ak \cite{OSZ}, separable reflexive asymptotic $\ell_p$ spaces (understood in that general sense) share nice structural properties, {\it e.g.} they embed in reflexive spaces with asymptotic $\ell_p$ FDD's. For more information on this topic, the reader is referred to \cite{OSZ} and the references therein.

At this point, let us mention a~result due to Knaust, Odell and Schlumprecht which reveals the connection between $\ell_1$-asymptoticity and summability of Szlenk index.
\begin{theorem}[{\it cf. }{\cite[Prop.~6.7]{kos}}]\label{kos_theorem}
Let $(E_k)_{k=1}^\infty$ be an FDD for a~Banach space $X$. Then $X$ has summable $H$-index with respect to $(E_k)_{k=1}^\infty$ if and only if there exists a~blocking $(H_j)_{j=1}^\infty$ of $(E_k)_{k=1}^\infty$ which is skipped asymptotic $\ell_1$, that is, for some positive constant $c$ and every skipped block sequence $(x_j)_{j=1}^n$ of $(H_j)_{j=n}^\infty$ we have
\begin{equation*}
\Biggl\|\sum_{j=1}^nx_j\Biggr\|\geq c\sum_{j=1}^n\n{x_j}.
\end{equation*}
\end{theorem}

For a precise definition of $H$-index and its summability, see \cite[\S 2~and \S 6]{kos}. Let us note that in the case where the FDD in question is given by an~unconditional basis, then arguing in a~similar way as in the proof of Theorem~\ref{E-sum} we infer that summability of the Szlenk index for $X$ is the same as summability of the $H$-index for $X^\ast$, and this is in turn equivalent to saying that $p(X)=1$ in the strong sense---namely, that the infimum in \eqref{p_def} is attained ({\it cf. }\cite[Prop.~6.9]{kos}). Hence, in our circumstances, it is most natural to require that $\EE^\ast$ is asymptotic $\ell_p$ with some $p$.

\begin{lemma}\label{type>=p}
If $\EE$ is a Banach space with a~normalized, shrinking, unconditional basis $(e_n)_{n=1}^\infty$ such that for some $p\in [1,\infty)$ its dual $\EE^\ast$ is asymptotic $\ell_p$ with respect to $(e_n^\ast)_{n=1}^\infty$, then $p(\EE)=p$.
\end{lemma}
\begin{proof}
Let $C\geq 1$ be so that $\EE^\ast$ is $C$-asymptotic $\ell_p$, let $\e\in (0,\frac{1}{C})$ and $N=\lfloor (C\e)^{-p}\rfloor$. Consider the~tree-map $(x_a^\ast)_{a\in S}$ in $\EE^\ast$ of height $N+1$ given by
$$
x_{a\smallfrown k}^\ast=\e e_{kN+\abs{a}}^\ast\quad\mbox{for each }a\in S\mbox{ with }\abs{a}<N\mbox{ and }k>\max a.
$$
Plainly, it is a~weak$^\ast$-null tree-map satisfying $\n{x_a^\ast}=\e$ for every $a\in S$ with $1\leq\abs{a}\leq N$. Moreover, if $\beta\subset S$ is a~branch, then by using the upper estimate in the definition of $\ell_p$-asymptoticity we obtain $\n{\sum_{a\in\beta}x_a^\ast}\leq \e\,CN^{1/p}\leq 1$. Thus, by Lemma~\ref{gkl_lemma} we infer that $\Sz(\EE,\e)>\lfloor (C\e)^{-p}\rfloor$. This shows that the Szlenk power type of $\EE$, if exists (that is, if $\Sz(\EE)\leq\omega$), must be at least equal to $p$. The reverse inequality will follow from Theorem~\ref{power} below.
\end{proof}

Now, our goal is to derive a~log-type estimate in asymptotic $\ell_p$ spaces, which will be crucial for the proof of Theorem~\ref{power}. Proposition~\ref{asymp} below already appeared in the literature in various forms ({\it e.g.} for Tsirelson's space it corresponds to the results in \cite[Ch.~4]{CS}). The key part of the proof is a~disjointization lemma (see Step~2 below)---first showed in \cite{PR} for $L^p$-spaces, then simplified by Kwapie\'n, and finally extended to general Banach lattices by W.B.~Johnson ({\it cf. }\cite[Prop.~3]{LT}). For us, positivity of the resulting operator is of vital importance, so we reproduce the argument carefully.

We consider the well-known {\it fast growing hierarchy}, that is, the sequence of functions $g_k\colon\N\to\N$ given by
$g_0(n)=n+1$ and $g_{k+1}(n)=g_k^{(n)}(n)$ (the $n$-fold iteration) for $k\geq 0$. Notice that $$
g_1(n)=2n,\,\,\, g_2(n)=n\cdot 2^n\,\,\,\mbox{and}\,\,\, g_3(n)\geq 2^{2^{\iddots^{{}_{2^n}}}}\,\,\mbox{(}n\mbox{ occurrences of }2\mbox{)}.
$$

\begin{proposition}\label{asymp}
Let $X$ be a Banach space with a $1$-unconditional basis $(e_n)_{n=1}^\infty$ with respect to which it is $C$-asymptotic $\ell_p$, for some $1\leq p\leq \infty$ and $C\geq 1$. Then every $n$-dimensional subspace of $X$ spanned by vectors supported after $[1,n]$ is $3C^8$-isomorphic (in the Banach--Mazur sense) via a~positive operator to a~subspace of $\ell_p^N$ with $N\leq (4n^2)^n$.
\end{proposition}
\begin{proof}
We shall deal with the case $p<\infty$; the proof works perfectly well also for $p=\infty$. We split the reasoning into three steps.

\vspace*{2mm}\noindent
{\sc Step 1. }We {\it claim} that for every $k\geq 0$ any $g_k(n)$ normalized block vectors in $X$, supported after $[1,n]$ are $C^{k+1}$-equivalent to the unit vector basis of $\ell_p^{g_k(n)}$.

We proceed by induction on $k$. If $k=0$, we have $n+1$ block vectors supported after $[1,n]$, hence the assertion follows by definition.

Now, suppose the assertion is valid for some $k\geq 0$ and let $(v_i\colon 1\leq i\leq g_{k+1}(n))$ be a~normalized block basic sequence of $(e_m)_{m=1}^\infty$ with $$n<\supp(v_1)<\ldots<\supp(v_{g_{k+1}(n)}).$$Let
$$
E_j=\bigl\{g_k^{(j-1)}(n)+1,g_k^{(j-1)}(n)+2,\ldots ,g_k^{(j)}(n)\bigr\}\quad\mbox{for }1\leq j\leq n,
$$
where we adopt the convention that $g_k^{(0)}(n)=0$. Thus, $E_j$'s partition the set of indices $\{1,\ldots,g_{k+1}(n)\}$ into $n$ consecutive pieces. Since $X$ is $C$-asymptotic $\ell_p$, for any sequence of scalars $(a_i\colon 1\leq i\leq g_{k+1}(n))$ we have 
\begin{equation}\label{avi}
\Biggl\|\sum_{i=1}^{g_{k+1}(n)}a_iv_i\Biggr\|\geq\frac{1}{C}\Biggl(\sum_{j=1}^n\Biggl\|\sum_{i\in E_j}a_iv_i\Biggr\|^p\Biggr)^{\! 1/p}.
\end{equation}
Fix, for a moment, any $j\in\{1,\ldots,n\}$ and observe that the elements in $(v_i\colon i\in E_j)$ satisfy
$$
g_k^{(j-1)}(n)<\supp(v_{\min E_j})<\ldots<\supp(v_{\max E_j}).
$$
Moreover, $\abs{E_j}\leq g_k^{(j)}(n)=g_k(g_k^{(j-1)}(n))$, whence the inductive hypothesis gives 
$$
\Biggl\|\sum_{i\in E_j}a_iv_i\Biggr\|\geq \frac{1}{C^{k+1}}\Biggl(\sum_{i\in E_j}\abs{a_i}^p\Biggr)^{\!\! 1/p}\quad\mbox{for each }1\leq j\leq n.
$$
Therefore, by \eqref{avi} we obtain
$$
\Biggl\|\sum_{i=1}^{g_{k+1}(n)}a_iv_i\Biggr\|\geq\frac{1}{C}\Biggl(\frac{1}{C^{p(k+1)}}\sum_{j=1}^n\sum_{i\in E_j}\abs{a_i}^p\Biggr)^{\!\! 1/p}=\frac{1}{C^{k+2}}\Biggl(\sum_{i=1}^{g_{k+1}(n)}\abs{a_i}^p\Biggr)^{\!\! 1/p}
$$
which is the desired lower estimate. The upper estimate is derived similarly.

\vspace*{2mm}\noindent
{\sc Step 2. }For any $n\in\N$ and $\e>0$ we define $N(n,\e)=\lceil 2n^2/\e\rceil^n$. Let $L$ be a Banach lattice and $F\subset L$ be an $n$-dimensional subspace. For every $\e>0$ there exist pairwise disjoint elements $\{g_i\}_{i=1}^{N(n,\e)}$ of $L$ and a~positive linear operator $V\colon F\to G=\mathrm{span}\{g_i\}_{i=1}^{N(n,\e)}$ so that $\n{Vx-x}\leq\e\n{x}$ for each $x\in F$.

\vspace*{2mm}
We give an outline of the argument given in \cite[Prop.~3]{LT}. First, take an Auerbach basis $\{f_i\}_{i=1}^n$ of $F$, so that for any sequence of scalars $(a_i)_{i=1}^n$ we have $\n{\sum_{i=1}^na_if_i}\geq\max_{1\leq i\leq n}\abs{a_i}$. Set $f_0=n^{-1}\sum_{i=1}^n\abs{f_i}$, where $\abs{f}=f\vee (-f)$, and define a~Banach lattice $(Z,\vertiii{\cdot})$ by
$$
Z=\bigl\{f\in L\colon\abs{f}<tf_0\mbox{ for some }t>0\bigr\},\quad \vertiii{f}=\inf\bigl\{t>0\colon\abs{f}<tf_0\bigr\}.
$$
Then $Z$ is an abstract $M$-space with the strong unit $f_0$, whence by Kakutani's theorem ({\it cf. }\cite[Thm.~1.b.6]{lt}) it is order isometric to a~sublattice of a~certain $L^\infty(\Omega)$.

Set $d=\lceil 2n^2/\e\rceil$ and pick pairwise disjoint intervals $I_1,\ldots,I_d\subseteq [-1,1]$ covering the whole of $[-1,1]$, each of length at most $\e/n^2$. For any $\gamma=(i_1,\ldots,i_n)\in [d]^n$ define
$$
G_\gamma=\bigl\{\omega\in\Omega\colon f_j(\omega)\in I_{i_j}\mbox{ for each }1\leq j\leq n\bigr\}.
$$
Then $\{G_\gamma\colon \gamma\in [d]^n\}$ is a~collection of at most $d^n=N(n,\e)$ non-empty mutually disjoint measurable sets, hence the sublattice of all linear combinations of the characteristic functions of all $G_\gamma$'s is order isometric to $\ell_\infty^m$ with $m\leq d^n$.

Now, we approximate each $f_i$ by a~linear combination of $\ind_{G_\gamma}$'s to within $\e/n^2$. To this end, for each $\gamma\in [d]^n$ pick any point $\omega_\gamma\in G_\gamma$ (we ignore those $\gamma$'s for which $G_\gamma=\varnothing$) and set $t_{\gamma,i}=f_i(\omega_\gamma)$. Then, plainly,
\begin{equation}\label{V1}
\vertiiib{f_i-\sum_{\gamma\in [d]^n}t_{\gamma,i}\ind_{G_\gamma}}\leq\frac{\e}{n^2}\quad\mbox{for each }1\leq i\leq n.
\end{equation}
For every $x\in F$, $x=\sum_{i=1}^na_if_i$ define
\begin{equation}\label{V2}
Vx=\sum_{i=1}^na_i\!\sum_{\gamma\in [d]^n}t_{\gamma,i}\ind_{G_\gamma}
\end{equation}
and note that if $x\geq 0$, then for every $\omega\in\Omega$ by picking the unique $\gamma$ with $\omega\in G_\gamma$ we obtain
$$
(Vx)(\omega)=\sum_{i=1}^na_it_{\gamma,i}=\sum_{i=1}^na_if_i(\omega_\gamma)=x(\omega_\gamma)\geq 0.
$$
This shows that $V$ is a~positive linear operator. Moreover, if $\vertiii{x}\leq n$, then \eqref{V1} and \eqref{V2} give $\vertiii{Vx-x}\leq\e$ which means that $\abs{Vx-x}\leq\e f_0$. Therefore, $\n{Vx-x}\leq\e$ for every $x\in F$ with $\n{x}\leq 1$.

\vspace*{2mm}\noindent
{\sc Step 3. }Fix any $n$-dimensional subspace $F$ of $X$ spanned by vectors supported after $[1,n]$. Since $X$ is equipped with the $1$-unconditional basis, it has the natural structure of Banach lattice, so we can apply the assertion of Step~2 to $L=X$ and $\e=\frac{1}{2}$. Observe also that $N(n,\frac{1}{2})=(4n^2)^n\leq g_3(n)$, whence by applying the claim proved in Step~1 (with $k=3$) we obtain an into-isomorphism $\Phi=W\circ V\colon F\to \ell_p^N$, where $N\leq (4n^2)^n$ and $W$ is given by $Wg_i=e_i$ for $1\leq i\leq (4n^2)^n$. Since $V$ was positive, so is $\Phi$. Finally,
$$
\n{\Phi}\n{\Phi^{-1}}\leq\n{V}\n{V^{-1}}\n{W}\n{W^{-1}}\leq \frac{3}{2}\cdot 2\cdot C^4\cdot C^4=3C^8,
$$
which completes the proof.
\end{proof}

We are prepared to prove the aforementioned log-type estimate, but first let us introduce a~bit of notation. If $(e_n)_{n=1}^\infty$ is a~$1$-unconditional basis of $X$, then given any $x,y\in X$ we set
$$
x\star_{\ssp}y=\sum_{n=1}^\infty \Bigl(\abs{e_n^\ast(x)}^p+\abs{e_n^\ast(y)}^p\Bigr)^{\! 1/p}\! e_n\in X,
$$
which defines an associative operation in $X$. Note that the so-defined element coincides with 
$(\abs{x}^p+\abs{y}^p)^{1/p}$ given by the Yudin--Krivine functional calculus ({\it cf. }\cite[\S 1.d]{lt}); we shall use their notation henceforth. (As above, $x\in X$ is called positive if $e_n^\ast(x)\geq 0$ for each $n\in\N$.) We will need the following lemma ({\it cf. }\cite[Prop.~1.d.9]{lt}).

\begin{lemma}[{\bf Krivine's inequalities}]\label{krivine}
Let $X$ and $Y$ be Banach lattices and let $\Phi\colon X\to Y$ be a~positive operator. Then for every sequence $(x_i)_{i=1}^n\subset X$ we have:
$$
\Biggl\|\Biggl(\sum_{i=1}^n\abs{\Phi (x_i)}^p\Biggr)^{\!\!\!1/p}\Biggr\|\leq\n{\Phi}\Biggl\|\Biggl(\sum_{i=1}^n\abs{x_i}^p\Biggr)^{\!\!\!1/p}\Biggr\|\quad\mbox{if }\,1\leq p<\infty
$$
and
$$
\Biggl\|\,\bigvee_{i=1}^n\abs{\Phi(x_i)}\,\Biggr\|\leq\n{\Phi}\Biggl\|\,\bigvee_{i=1}^n\abs{x_i}\,\Biggr\|\quad\mbox{if }\,p=\infty.
$$
\end{lemma}

\begin{proposition}\label{log_estimate}
Let $X$ be a Banach space with a $1$-unconditional basis $(e_n)_{n=1}^\infty$ with respect to which it is $C$-asymptotic $\ell_p$, for some $1\leq p\leq \infty$ and $C\geq 1$. Then there is a~constant $B>0$ depending only on $C$ such that for every $n\geq 2$ and all positive vectors $x_1,\ldots,x_n\in X$ we have
$$
\n{x_1\star_{\ssp}\ldots\star_{\ssp}x_n}\geq\frac{\bigl(\n{x_1}^p+\ldots+\n{x_n}^p\bigr)^{\! 1/p}}{B(\log n)^2}.
$$
\end{proposition}
\begin{proof}
We write down the proof assuming that $p<\infty$, as the case $p=\infty$ is trivial. We start with noticing that recursive application of inequality \eqref{asymp_def} leads to the fact that for a~suitable constant $A>0$ and every $n\geq2$ we have
\begin{equation}\label{logg}
\frac{1}{A\log n}\n{x}_{\ell_p^n}\leq\n{x}\leq (A\log n)\n{x}_{\ell_p^n}\quad\mbox{whenever }x\in\mathrm{span}\{e_i\}_{i=1}^n
\end{equation}
(just apply \eqref{asymp_def} to the vectors $\sum_{2^{-j}n<i\leq 2^{-j+1}n}e_i^\ast(x)e_i$ with $1\leq j\leq\lceil\log_2n\rceil$).

Set $z=x_1\star_{\ssp}\ldots\star_{\ssp}x_n$. We will prove our assertion by estimating separately the norms $\n{P_{[1,n]}z}$ and $\n{P_{(n,\infty)}z}$. Since splitting into two parts is allowed only for vectors supported on $[2,\infty)$, we shall first deal with the case where $e_1^\ast(x_i)$'s are relatively large for enough many $i$'s. So, let us define
$$
I=\Biggl\{1\leq i\leq n\colon e_1^\ast(x_i)\geq\frac{1}{2}\n{x_i}\Biggr\}\,\,\mbox{ and }\,\, J=\{j_1,\ldots,j_k\}=\{1,\ldots,n\}\setminus I.
$$
Note that the norm of $z$ is at least equal to
\begin{equation*}
\Biggl\|\Biggl(\sum_{i\in I}e_1^\ast(x_i)^p\Biggr)^{\!\!\! 1/p}\!\!\!e_1+P_{[2,\infty)}(x_{j_1}\star_{\ssp}\ldots\star_{\ssp}x_{j_k})\Biggr\|\geq \frac{1}{2}\Biggl(\sum_{i\in I}\n{x_i}^p\Biggr)^{\!\!\! 1/p}\,\vee\,\bigl\|P_{[2,\infty)}(x_{j_1}\star_{\ssp}\ldots\star_{\ssp}x_{j_k})\bigr\|.
\end{equation*}
Of course, $\sum_{i\in I}\n{x_i}^p$ and $\sum_{j\in J}\n{x_j}^p$ cannot be simultaneously less than $\frac{1}{2}\sum_{i=1}^n\n{x_i}^p$. Therefore, in the rest of the proof we can (and we do) assume that $\supp(z)\subset [2,\infty)$. Obviously, we can also assume that $\supp(x_i)$ is finite for every $1\leq i\leq n$.

Since $X$ is $C$-asymptotic $\ell_p$, we have
\begin{equation}\label{z_estimate}
\n{z}\geq\frac{1}{C}\Bigl(\bigl\|P_{[2,n]}z\bigr\|^p+\bigl\|P_{(n,\infty)}z\bigr\|^p\Bigr)^{\!\! 1/p}.
\end{equation}
Now, we shall estimate both these summands separately. Firstly, by \eqref{logg} we have 
\begin{equation}\label{log1}
\begin{split}
\bigl\|P_{[2,n]}z\bigr\|^p &\geq \frac{1}{(A\log n)^p}\bigl\|P_{[2,n]}z\bigr\|_{\ell_p^n}^p\\ &=\frac{1}{(A\log n)^p}\sum_{i=1}^n \bigl\|P_{[2,n]}x_i\bigr\|_{\ell_p^n}^p
\geq\frac{1}{(A\log n)^{2p}}\sum_{i=1}^n\bigl\|P_{[2,n]}x_i\bigr\|^p.
\end{split}
\end{equation}
Secondly, set $F\coloneqq\mathrm{span}\{P_{(n,\infty)}x_i\}_{i=1}^n$ and let $\Phi\colon F\to\ell_p^N$ be a~positive $3C^8$-isomorphism produced by Proposition~\ref{asymp}. Using Lemma~\ref{krivine} we obtain
\begin{equation*}
\begin{split}
\bigl\|P_{(n,\infty)}z\bigr\| &=\Biggl\|P_{(n,\infty)}\Biggl(\sum_{i=1}^n\abs{x_i}^p\Biggr)^{\!\! 1/p}\Biggr\|\\
&\geq \frac{1}{\n{\Phi}}\Biggl\|\Biggl(\sum_{i=1}^n\abs{\Phi(P_{(n,\infty)}x_i)}^p\Biggr)^{\!\! 1/p}\Biggr\|_{\ell_p^N}=
\frac{1}{\n{\Phi}}\Biggl(\sum_{i=1}^n\bigl\|\Phi(P_{(n,\infty)}x_i)\bigr\|_{\ell_p^N}^p\Biggr)^{\!\! 1/p},
\end{split}
\end{equation*}
where the last equality follows simply by changing the order of summation. Therefore,
\begin{equation}\label{log2}
\bigl\|P_{(n,\infty)}z\bigr\|\geq \frac{1}{\n{\Phi}\n{\Phi^{-1}}}\Biggl(\sum_{i=1}^n\bigl\|P_{(n,\infty)}x_i\bigr\|^p\Biggr)^{\!\! 1/p}\geq \frac{1}{3C^8}\Biggl(\sum_{i=1}^n\bigl\|P_{(n,\infty)}x_i\bigr\|^p\Biggr)^{\!\! 1/p}.
\end{equation}
Notice that \eqref{asymp_def} obviously implies 
\begin{equation}\label{uv}
\bigl\|P_{[2,n]}x_i\bigr\|^p+\bigl\|P_{(n,\infty)}x_i\bigr\|^p\geq\frac{\n{x_i}^p}{C^p}\quad\mbox{for every }1\leq i\leq n.
\end{equation}
Combining \eqref{z_estimate}, \eqref{log1}, \eqref{log2} and \eqref{uv} we obtain
\begin{equation*}
\begin{split}
\n{z} \geq\frac{1}{C\!\cdot\!\max\{3C^8, (A\log n)^2\}} \Biggl(\sum_{i=1}^n &\Bigl(\bigl\|P_{[2,n]}x_i\bigr\|^p+\bigl\|P_{(n,\infty)}x_i\bigr\|^p\Bigr)\Biggr)^{\!\! 1/p}\\
&\,\quad\geq \frac{\bigl(\n{x_1}^p+\ldots+\n{x_n}^p\bigr)^{\! 1/p}}{B(\log n)^2},
\end{split}
\end{equation*}
with an appropriate constant $B>0$.
\end{proof}

\begin{remark}
We have trivially $\n{x}\leq\n{x}_{\ell_1}$ for every $x\in X$, so in the case $p=1$ we do not need to square the logarithm ({\it cf. }inequality \eqref{logg}). For instance, what in fact has been proved for Tsirelson's space $\TT^\ast$ is the following estimate:
$$
\n{x_1+\ldots+x_n}\geq \frac{\n{x_1}+\ldots+\n{x_n}}{B\log n}
$$
for each $n\geq 2$ and all positive vectors $x_1,\ldots,x_n\in\TT^\ast$. The important task is, of course, that we do not assume anything about the supports of $x_i$'s.
\end{remark}

\begin{corollary}\label{pq-log_estimate}
Let $X$ be a Banach space with a $1$-unconditional basis $(e_n)_{n=1}^\infty$ with respect to which it is $C$-asymptotic $\ell_p$, for some $1\leq p\leq \infty$ and $C\geq 1$. There is a~constant $B>0$ depending only on $C$ such that for all $q\geq p$, $n\geq 2$ and all positive vectors $x_1,\ldots,x_n\in X$ we have
$$
\n{x_1\star_{\ssq}\ldots\star_{\ssq}x_n}\geq\frac{\bigl(\n{x_1}^p+\ldots+\n{x_n}^p\bigr)^{\! 1/p}}{Bn^{1/p-1/q}(\log n)^2}.
$$
\end{corollary}
\begin{proof}
By the inequality between power means, we have $(\frac{1}{n}\sum_{i=1}^na_i^q)^{1/q}\geq (\frac{1}{n}\sum_{i=1}^na_i^p)^{1/p}$ for all non-negative numbers $a_1,\ldots,a_n$. Therefore, each coordinate of $x_1\star_{\ssq}\ldots\star_{\ssq}x_n$ is not less than the corresponding coordinate of $x_1\star_{\ssp}\ldots\star_{\ssp}x_n$ divided by $n^{1/p-1/q}$. The assertion hence follows from Proposition~\ref{log_estimate}.
\end{proof}

\begin{theorem}\label{power}
Let $\EE$ be a Banach space with a normalized, shrinking, $1$-unconditional basis $(e_n)_{n=1}^\infty$ such that for some $p\in [1,\infty)$ its dual $\EE^\ast$ is asymptotic $\ell_p$ with respect to $(e_n^\ast)_{n=1}^\infty$. Then for every power type bounded sequence $(X_n)_{n=1}^\infty$ we have
$$
p\Bigl(\!\Bigl(\bigoplus_{n=1}^\infty X_n\Bigr)_{\!\!\EE}\Bigr)=\max\bigl\{p,\,\mathfrak{p}(X_n)_{n=1}^\infty\bigr\}.
$$
\end{theorem}
\begin{proof}
Set $X=(\bigoplus_{n=1}^\infty X_n)_\EE$ and fix any $\e\in (0,1)$, $N\in\N$, $N\geq 2$ so that $\iota_\e^N B_{X^\ast}\not=\varnothing$. Then, by Lemma~\ref{gkl_lemma}, there is a~weak$^\ast$-null tree-map $(x_a^\ast)_{a\in S}$ in $X^\ast$ such that $\n{x_a^\ast}\geq\frac{1}{4}\e$ for each $a\in S$ and $\n{\sum_{a\in\beta}x_a^\ast}\leq 1$ for each branch $\beta\subset S$. 

Fix any $q,r$ with $\mathfrak{p}(X_n)_{n=1}^\infty<r<q$; in the case where $\mathfrak{p}(X_n)_{n=1}^\infty<p$, we also require that $q<p$. In each case we have $p(X_n)<r<q$ for every $n\in\N$ and hence there exists a~common constant $B>0$ satisfying $\Sz(X_n,\e)\leq B\e^{-r}$ for all $\e\in (0,1)$ and $n\in\N$.

For $n\in\N$, let $\abs{\,\cdot\,}_n$ be a~norm produced by Proposition~\ref{p-renorming} when applied to the space $X_n$; the corresponding constant $\gamma=\gamma(B,q,r)$ is common for all $n$'s. Define
$$
\vertiii{x^\ast}=\Biggl\|\sum_{n=1}^\infty\left|x_n^\ast\right|_{n}\! e_n^\ast\Biggr\|\quad\mbox{for }x^\ast=(x_n^\ast)_{n=1}^\infty\in X^\ast,
$$
which is plainly a~$2$-equivalent dual norm on $X^\ast$, in fact, $\n{x^\ast}\leq\vertiii{x^\ast}\leq 2\n{x^\ast}$ for every $x^\ast\in X^\ast$. Likewise in the proof of Theorem~\ref{c0-sum}, we shall extract a~partial branch $\beta\subset S$ and force a~lower estimate for $\n{\sum_{a\in\beta}x_a^\ast}$, this time in terms of the $\ell_p$-norm of $(\n{x_a^\ast})_{a\in\beta}$.

For any $a\in S$ we denote by $a^+$ the set of all successors of $a$. If $\abs{a}<N$ and $M\subseteq a^+$ is an infinite set, we define
$$
\nu(a,M)=\sup_{b\in M}\min\Biggl\{n\in\N\colon \bigl\|P_{[1,n]}x_b^\ast\bigr\|\geq\frac{1}{8}\e\Biggr\}.
$$
We call $a$ {\it of type I} if there exists an infinite set $M\subseteq a^+$ so that $\nu(a,M)<\infty$, and in this case we fix one such set and call it $M_a$. We say that $a$ is {\it of type II} if it is not of type I. 

We are going to define a~partial branch $\beta=(a_0,a_1,\ldots ,a_N)$ by induction. Set $a_0=\varnothing$. Given $0\leq k<N$, suppose we have already selected $a_0,a_1,\ldots,a_k\in S$, together with positive, finitely supported vectors $v_1,\ldots,v_k\in\EE^\ast$ and finite intervals $I_1,\ldots,I_k\subset\N$ so that the following conditions are satisfied:
\begin{equation}\label{p0}
\vertiii{x_{a_j}^\ast-\w x_{a_j}^\ast}<\frac{1}{N}\quad\mbox{for each }1\leq j\leq k,\,\mbox{ where }\w x_{a_j}^\ast\coloneqq P_{I_j}x_{a_j}^\ast,
\end{equation}
\begin{equation}\label{p1}
\Biggl|P_i\Biggl(\sum_{j=1}^k\w x_{a_j}^\ast\Biggr)\Biggr|_i> \Bigl(\frac{\gamma}{2}\Bigr)^{\! 1/q}e_i^\ast(v_1\star_{\ssq}\ldots\star_{\ssq}v_k)\quad\mbox{for each }i\in\bigcup_{j=1}^k\supp(v_j)
\end{equation}
and
\begin{equation}\label{p2}
\n{v_j}\geq\frac{1}{48}\e\quad\mbox{for each }1\leq j\leq k.
\end{equation}
(So, for $k=0$ we do not assume anything.) Let us consider two possibilities.

\vspace*{2mm}\noindent
{\it Case 1. }$a_k$ is of type I. 

\vspace*{1mm}\noindent
Let $\nu(a_k)=\nu(a_k,M_{a_k})$. Then for every $b\in M_{a_k}$ we have $\n{P_{[1,\nu(a_k)]}x_b^\ast}\geq\frac{1}{8}\e$. By passing to an~infinite subset $M\subseteq M_{a_k}$ we may assume that there exist non-negative numbers $\tau_{k+1,i}$ with $1\leq i\leq \nu(a_k)$ such that for every $b\in M$ we have 
$$
\tau_{k+1,i}\leq\abs{P_ix_b^\ast}_i\leq 2\tau_{k+1,i}\quad\mbox{for each }1\leq i\leq\nu(a_k)\mbox{ with }\lim_{b\in M}\abs{P_ix_b^\ast}_i>0,
$$
whereas $\tau_{k+1,i}=0$ for all other $i$'s. Define a positive vector $v_{k+1}\in\EE^\ast$ by
$$
v_{k+1}=(\tau_{k+1,1},\ldots,\tau_{k+1,\nu(a_k)},0,0,\ldots).
$$
We can pass, if necessary, to another infinite subset of $M$ (still denoted by $M$) and assume that $\n{v_{k+1}}\geq\frac13\n{P_{[1,\nu(a_k)]}x^\ast_b}$. Now, we shall use properties of the norms $\abs{\,\cdot\,}_i$ for each $1\leq i\leq\max\bigcup_{j=1}^{k+1}\supp(v_j)$; for any such $i$ we apply Proposition~\ref{p-renorming} to $x^\ast=P_i\sum_{j=1}^k\w x_{a_j}^\ast\in X_i^\ast$ and the weak$^\ast$-null sequence $(P_ix_b^\ast)_{b\in M}\subset X_i^\ast$. As a~result, we obtain a~successor $a_{k+1}\in M$ of $a_k$ such that 
\begin{equation*}
\begin{split}
\Biggl|P_i\Biggl(\sum_{j=1}^{k}\w x_{a_j}^\ast &+x_{a_{k+1}}^\ast\Biggr)\Biggr|_i> \Biggl[\Biggl|P_i\Biggl(\sum_{j=1}^{k}\w x_{a_j}^\ast\Biggr)\Biggr|_i^q+\frac{\gamma}{2}(e_i^\ast(v_{k+1}))^q\Biggr]^{\! 1/q}\\
& \geq \Bigl(\frac{\gamma}{2}\Bigr)^{1/q}\Bigl(e_i^\ast(v_1\star_{\ssq}\ldots\star_{\ssq}v_k)^q+\tau_{k+1,i}^q\Bigr)^{\! 1/q}\\
& =\Bigl(\frac{\gamma}{2}\Bigr)^{1/q} e_i^\ast(v_1\star_{\ssq}\ldots\star_{\ssq}v_{k+1})\quad\mbox{for each }i\in\N\mbox{ with }e_i^\ast(v_{k+1})>0,
\end{split}
\end{equation*}
whereas for other values of $i$ for which $e_i^\ast(v_1\star_{\ssq}\ldots\star_{\ssq}v_k)>0$, the inequality above can be guaranteed by the induction hypothesis. Pick any finite interval $I_{k+1}\subset\N$ so that condition \eqref{p0} holds true for $j=k+1$ and the inequalities above remain valid after replacing $x_{a_{k+1}}^\ast$ by $\w x_{a_{k+1}}^\ast$. Observe also that 
$$
\n{v_{k+1}}\geq \frac{1}{6} \vertiii{P_{[1,\nu(a_k)]}x_{a_{k+1}}^\ast}\geq \frac{1}{48}\e.
$$

\vspace*{2mm}\noindent
{\it Case 2. }$a_k$ is of type II.

\vspace*{1mm}\noindent
In this case, we can choose $a_{k+1}\in a_k^+$ for which there exists a~finite interval $I\subset\N$ satisfying the following conditions:
\begin{itemize}[leftmargin=7mm]
\setlength\itemsep{.4em}
\item $\max I_j<\min I$ for every $1\leq j\leq k$,
\item $\max\supp(v_1\star_{\ssq}\ldots\star_{\ssq}v_{k})<I$,
\item $\bigl\|P_Ix_{a_{k+1}}^\ast\bigr\|\geq\frac{1}{8}\e$
\end{itemize}
and, moreover, 
\begin{equation}\label{p3}
\displaystyle{\Biggl|P_i\Biggl(\sum_{j=1}^{k}\w x_{a_j}^\ast+x_{a_{k+1}}^\ast\Biggr)\Biggr|_i>\Bigl(\frac{\gamma}{2}\Bigr)^{1/q} e_i^\ast(v_1\star_{\ssq}\ldots\star_{\ssq}v_k)}\quad\mbox{ for every }i\in\bigcup_{j=1}^k\supp(v_j).
\end{equation}
The last condition can be guaranteed in view of \eqref{p1} and the fact that the sequence $(x_b^\ast)_{b\in a^+}$ is weak$^\ast$-null. Define a~positive vector $v_{k+1}\in\EE^\ast$ by
$$
v_{k+1}=P_{I}\bigl(\bigl|P_ix_{a_{k+1}}^\ast\bigr|_i\bigr)_{i=1}^\infty
$$
(the outer projection symbol corresponds to the basis $(e_n^\ast)_{n=1}^\infty$ of $\EE^\ast$, while the inner to the direct sum decomposition of $X$) and pick a~finite interval $I_{k+1}\subset\N$ so that $I\subseteq I_{k+1}$, condition \eqref{p0} is valid for $j=k+1$ and inequality \eqref{p3} remains true when $x_{a_{k+1}}^\ast$ is replaced by  $\w x_{a_{k+1}}^\ast$. For $i\in I$, we obviously have
$$
\Biggl|P_i\Biggl(\sum_{j=1}^{k+1}\w x_{a_j}^\ast\Biggr)\Biggr|_i=\bigl|P_ix_{a_{k+1}}^\ast\bigr|_i=e_i^\ast(v_{k+1})=e_i^\ast(v_1\star_{\ssq}\ldots\star_{\ssq}v_{k+1}),
$$
whence by \eqref{p3} we obtain condition \eqref{p1} for $k+1$ in the place of $k$. Condition \eqref{p2} for $i=k+1$ holds true automatically. 

As a result, we obtain a partial branch $\beta=(a_0,a_1,\ldots,a_N)$ such that 
$$
\Biggl|P_i\Biggl(\sum_{a\in\beta}\w x_a^\ast\Biggr)\Biggr|_i\geq\Bigl(\frac{\gamma}{2}\Bigr)^{1/q}e_i^\ast(v_1\star_{\ssq}\ldots\star_{\ssq}v_n)\quad\mbox{for every }i\in\N
$$
and $\n{v_n}\geq\frac{1}{48}\e$ for every $1\leq n\leq N$. Hence, in the case where $p<q$, Corollary~\ref{pq-log_estimate} yields
\begin{equation*}
\begin{split}
3\geq \vertiii{\sum_{a\in\beta}\w x_a^\ast} &=\Biggl\|\Biggl(\Biggl|P_i\sum_{a\in\beta}\w x_a^\ast\Biggr|_i   \Biggr)_{\!\!i=1}^{\!\!\infty}\Biggr\|_{\EE^\ast} \geq\Bigl(\frac{\gamma}{2}\Bigr)^{1/q}\n{v_1\star_{\ssq}\ldots\star_{\ssq}v_N}_{\EE^\ast}\\
& \geq \Bigl(\frac{\gamma}{2}\Bigr)^{1/q}\frac{\bigl(\n{v_1}^p+\ldots+\n{v_N}^p\bigr)^{\! 1/p}}{BN^{1/p-1/q}(\log N)^2}\\
&\geq\frac{1}{48}\!\cdot\!\Bigl(\frac{\gamma}{2}\Bigr)^{\!1/q}\!\cdot\!\frac{N^{1/p}\e}{BN^{1/p-1/q}(\log N)^2}=\frac{1}{48}\!\cdot\!\Bigl(\frac{\gamma}{2}\Bigr)^{\!1/q}\!\!\cdot\!\frac{N^{1/q}\e}{B(\log N)^2}.
\end{split}
\end{equation*}
In the case where $p>q$ we have $\n{v_1\star_{\ssq}\ldots\star_{\ssq}v_N}\geq\n{v_1\star_{\ssp}\ldots\star_{\ssp}v_N}$ and hence by Proposition~\ref{log_estimate} we obtain
$$
3\geq\frac{1}{48}\!\cdot\!\Bigl(\frac{\gamma}{2}\Bigr)^{\!1/q}\!\!\cdot\!\frac{N^{1/p}\e}{B(\log N)^2}.
$$
Therefore, in both cases we have
$$
\frac{N^{\min\left\{1/p,1/q\right\}}}{(\log N)^2}\leq 144 B\Bigl(\frac{2}{\gamma}\Bigr)^{\!1/q}\e^{-1}.
$$
Since $\log N=o(N^\alpha)$ for each $\alpha>0$, the above inequality implies that for any exponent $s>\max\{p,q\}$ there is a~constant $C_s>0$ so that $N\leq C_s\e^{-s}$. We have thus proved that $p(X)\leq \max\{p,q\}$ and since $q$ could be taken  arbitrarily close to $\mathfrak{p}(X_n)_{n=1}^\infty$, we have $p(X)\leq\max\{p,\mathfrak{p}(X_n)_{n=1}^\infty\}$. The reverse inequality follows from Lemma~\ref{type>=p}.
\end{proof}

\begin{corollary}
For every separable Banach space $X$ with $\Sz(X)=\omega$ and any $p\in (1,\infty)$ we have $p(\ell_p(X))=\max\{q,p(X)\}$, where $p^{-1}+q^{-1}=1$. Similarly, $p(c_0(X))=p(X)$.
\end{corollary}
\begin{corollary}
The space $\TT(c_0)$ has Szlenk power type $1$ but does not have summable Szlenk index.
\end{corollary}
\begin{proof}
Since $\TT^\ast$ is $2$-asymptotic $\ell_1$ with respect to its canonical basis, Theorem~\ref{power} applies. That $\TT(c_0)$ does not have summable Szlenk index was already said in Example~\ref{T(c0)}.
\end{proof}
\begin{remark}
As we have already mentioned (see the comments after Theorem~\ref{kos_theorem}), \cite[Prop.~6.9]{kos} implies that if a~Banach space $X$ has an unconditional shrinking basis, then $X$ has summable Szlenk index whenever there exists a~constant $K>0$ so that $\Sz(X,\e)\leq K\e^{-1}$ for every $\e\in (0,1)$. Hence, the example $X=\TT(c_0)$ in fact shows that it is possible that the infimum in \eqref{p_def} is not attained.
\end{remark}

As we shall see below, Theorem~\ref{power} may collapse drastically if we merely assume that $\EE$ is an asymptotic $\ell_p$ space with respect to a~general FDD.
\begin{example}
Let $\EE=(\bigoplus_{n=1}^\infty\ell_2^n)_{c_0}$ which is isomorphic to a~subspace of $c_0$ (and its dual is asymptotic $\ell_1$). Then the $\EE$-direct sum $\EE(c_0)$ of infinitely many copies of $c_0$ fails to have the Szlenk power type $1$. Indeed, for any $\e>0$ take a~natural number $n\leq\e^{-2}$ and consider a weak$^\ast$-null tree-map $(x_a^\ast)_{a\in S}$ in $(c_0\oplus\ldots\oplus c_0)_{\ell_2^n}^\ast$ of height $n+1$ constructed in such a~way that for every $a\in S$ with $\abs{a}<n$ we have
$$
x_{a\smallfrown m}^\ast=\bigl(\underbrace{0,\ldots,0}_{\abs{a}\text{ times}},\xi_{\abs{a}+1,m},0,\ldots,0\bigr)\quad\mbox{for }m>\max a,
$$
where $(\xi_{\abs{a}+1,m})_{m=1}^\infty$ is weak$^\ast$-null in $\ell_1$ and satisfies $\n{\xi_{\abs{a}+1,m}}=\e$ for each $m\in\N$. Then for every branch $\beta\subset S$ we have $\n{\sum_{a\in\beta}x_a^\ast}=\e\sqrt{n}\leq 1$. This shows that $\Sz(\EE(c_0),\e)>\e^{-2}$ which proves our claim. In fact, since $\EE(c_0)$ is naturally isometric to a~subspace of $c_0(\ell_2(c_0))$, we have $p(\EE(c_0))=2$. However, things can get much worse than that.
\end{example}

\begin{example}
Let $p_n=n/(n-1)$ and $\EE=(\bigoplus_{n=1}^\infty\ell_{p_n}^n)_{c_0}$ which is isomorphic to a~subspace of $c_0$. Then the $\EE$-direct sum $\EE(c_0)$ of infinitely many copies of $c_0$ has the Szlenk index $\omega^2$. Indeed, for any natural number $n$ consider a weak$^\ast$-null tree-map $(x_a^\ast)_{a\in S}$ in $(c_0\oplus\ldots\oplus c_0)_{\ell_{p_n}^n}^\ast$ of height $n+1$ constructed in such a~way that for every $a\in S$ with $\abs{a}<n$ we have
$$
x_{a\smallfrown m}^\ast=\bigl(\underbrace{0,\ldots,0}_{\abs{a}\text{ times}},\xi_{\abs{a}+1,m},0,\ldots,0\bigr)\quad\mbox{for }m>\max a,
$$
where $(\xi_{\abs{a}+1,m})_{m=1}^\infty$ is weak$^\ast$-null in $\ell_1$ and satisfies $\n{\xi_{\abs{a}+1,m}}=1/2$ for each $m\in\N$. Then for every branch $\beta\subset S$ we have $\n{\sum_{a\in\beta}x_a^\ast}=\sqrt[n]{n}/2\leq 1$. This yields $\Sz(\EE(c_0),1/2)>n$ for each $n\in\N$ and hence $\Sz(\EE(c_0),1/2)>\omega$. On the other hand, $\Sz(\EE(c_0))\leq\omega^2$ in view of Causey's result \cite[Theorem 5.14]{causey}.
\end{example}

\subsection*{Acknowledgements}
The first-named author was supported by the University of Silesia Mathematics Department (Iterative Functional Equations and Real Analysis program). The second-named author was supported by the Polish Ministry of Science and Higher Education in the years 2013-14, under Project No.~IP2012011072.

\bibliographystyle{amsplain}

\begin{thebibliography}{99}
\bibitem{ak} F. Albiac, N.J. Kalton, \emph{Topics in Banach space theory}, Graduate Texts in Mathematics~233, Springer, New York~2006.

\bibitem{brooker} P.A.H. Brooker, \emph{Direct sums and the Szlenk index}, J.~Funct. Anal.~{\bf 260} (2011), 2222--2246. 

\bibitem{ccky} F. Cabello S\'anchez, J.M.F. Castillo, N.J.~Kalton, D.T.~Yost, \emph{Twisted sums with $C(K)$ spaces}, Trans. Amer. Math. Soc.~{\bf 355} (2003), 4523--4541.

\bibitem{CS} P.G. Casazza, Th.J. Shura, \emph{Tsirelson's space}, Lecture Notes in Mathematics~1363, Springer-Verlag, Berlin--Heidelberg~1989.

\bibitem{causey} R. Causey, \emph{Concerning the Szlenk index}, {\tt arXiv:1501.06885 [math.FA]}.

\bibitem{figiel_johnson} T. Figiel, W.B. Johnson, \emph{A~uniformly convex Banach space which contains no $\ell_p$}, Compositio Math.~{\bf 29} (1974), 179--190.

\bibitem{gkl} G. Godefroy, N.J. Kalton, G. Lancien, \emph{Szlenk indices and uniform homeomorphisms}, Trans. Amer. Math. Soc.~{\bf 353} (2001), 3895--3918.

\bibitem{gkl_gafa} G. Godefroy, N. Kalton, G.~Lancien, \emph{Subspaces of $c_0(\N)$ and Lipschitz isomorphisms}, GAFA, Geom. Funct. Anal.~{\bf 10} (2000), 798--820.

\bibitem{kos} H. Knaust, E. Odell, Th. Schlumprecht, \emph{On asymptotic structure, the Szlenk index and UKK properties in Banach spaces}, Positivity~{\bf 3} (1999), 173--199.

\bibitem{lancien} G. Lancien, \emph{A~survey on the Szlenk index and some of its applications}, Rev. R.~Acad. Cien. Serie~A. Mat. {\bf 100} (2006), 209--235.

\bibitem{laustsen} N.J. Laustsen, \emph{Matrix multiplication and composition of operators on the direct sum of an infinite sequence of Banach spaces}, Math. Proc. Camb. Phil. Soc.~{\bf 131} (2001), 165--183.

\bibitem{LT} J. Lindenstrauss, L. Tzafriri, \emph{The uniform approximation property in Orlicz spaces}, Israel J.~Math.~{\bf 23} (1976), 142--155.

\bibitem{lt} J. Lindenstrauss, L. Tzafriri, \emph{Classical Banach spaces I (Sequence spaces) and II (Function spaces)}, Springer-Verlag, Berlin--Heidelberg--New York 1977.

\bibitem{MMT-J} B. Maurey, V.D. Milman, N.~Tomczak-Jaegermann, \emph{Asymptotic infinite-dimensional theory of Banach spaces}, Oper. Theory Adv. Appl.~{\bf 77} (1994), 149--175.

\bibitem{MT-J} V.D. Milman, N. Tomczak-Jaegermann, \emph{Asymptotic $\ell_p$ spaces and bounded distortions}, Contemp. Math.~{\bf 144} (1993), 173--195.

\bibitem{OSZ} E. Odell, Th. Schlumprecht, A. Zs\'ak, \emph{On the structure of asymptotic $\ell_p$ spaces}, Quart. J.~Math.~{\bf 59} (2008), 85--122.

\bibitem{PR} A. Pe\l czy\'nski, H.P. Rosenthal, \emph{Localization techniques in $L^p$ spaces}, Studia Math.~{\bf 52} (1975), 263--289.

\bibitem{szlenk} W. Szlenk, \emph{The non-existence of a~separable reflexive Banach space universal for all separable reflexive Banach spaces}, Studia Math.~{\bf 30} (1968), 53--61.
\end{thebibliography}

\end{document}